\documentclass[11pt, reqno]{amsart}

\usepackage{geometry}
\usepackage{hyperref}
\usepackage{mathtools}
\usepackage{tikz}
\usetikzlibrary{knots,calc}
\usepackage{mathrsfs}

\usepackage{cancel}
\usepackage{graphicx,verbatim,enumitem} 

\usepackage{ulem}
\usepackage{tikz}
\usetikzlibrary{arrows.meta,decorations.pathmorphing,calc}

\definecolor{linkpurple}{RGB}{120,40,140}
\definecolor{linkred}{RGB}{200,30,30}
\definecolor{linkgreen}{RGB}{0,120,80}
\definecolor{linkblue}{RGB}{40,100,220}

\usepackage{tikz-cd}

\usepackage[all,cmtip]{xy}

\let\olduplus\uplus
\renewcommand{\olduplus}{\pitchfork}

\usepackage{mathabx}
\usepackage{xcolor}
\usepackage{amssymb}

\usepackage{graphicx}
\usepackage[capitalise,noabbrev,nosort]{cleveref}

\DeclareFontFamily{OT1}{pzc}{}
\DeclareFontShape{OT1}{pzc}{m}{it}{<-> s * [1.10] pzcmi7t}{}
\DeclareMathAlphabet{\mathpzc}{OT1}{pzc}{m}{it}

\usepackage{amsthm}
\theoremstyle{plain}
\newtheorem{theorem}{Theorem}[section]
\newtheorem{lemma}[theorem]{Lemma}
\newtheorem*{lemma*}{Lemma}
\newtheorem{proposition}[theorem]{Proposition}
\newtheorem{corollary}[theorem]{Corollary}

\newtheorem{conjecture}[theorem]{Conjecture}

\newtheorem{terminology}[theorem]{Terminology}

\newcommand{\bd}{\partial}

\newcommand{\Z}{\mathbb{Z}}

\newcommand{\N}{\mathbb{N}}
\newcommand{\R}{\mathbb{R}}

\newcommand{\BBar}{\mathrm{Bar}}

\newcommand{\hConcord}{\mathsf{hConcord}}
\newcommand{\barDelta}{\widebar{\Delta}}
\newcommand{\into}{\hookrightarrow}

\newcommand{\green}[1]{\textcolor{green!60!black}{#1}}

\newcommand{\purple}[1]{\textcolor{purple}{#1}}
\newcommand{\blue}[1]{\textcolor{blue}{#1}}
\newcommand{\red}[1]{\textcolor{red}{#1}}
\newcommand{\greg}[1]{\textcolor{red}{[GBF: #1]}}
\newcommand{\nir}[1]{\textcolor{blue}{[NG: #1]}}

\newcommand{\dev}[1]{\textcolor{brown}{[DPS: #1]}}

\newcommand{\Primary}{\mathbb M}

\theoremstyle{definition}
\newtheorem{definition}[theorem]{Definition}

\newtheorem{to-do}[theorem]{To-Do}
\newtheorem{example}[theorem]{Example}

\newtheorem{remark}[theorem]{Remark}

\theoremstyle{remark}

\newtheoremstyle{Note}
  {\topsep}
  {}
  {\color{blue}}
  {0pt}
  {\bfseries\color{blue}}
  {:}
  { }
  {\thmname{#1}\thmnumber{ #2}\textnormal{\thmnote{\textbf{ (#3)}}}}

\theoremstyle{Note}

\newcommand{\xr}{\xrightarrow}
\newcommand{\id}{\textup{id}}
\newcommand{\Hom}{\textup{Hom}}

\title{
Bar cohomology of links: beyond Milnor invariants}

\author[G. Friedman]{Greg Friedman}
\address{Department of Mathematics, Texas Christian University, USA}
\email{\href{mailto:g.friedman@tcu.edu}{g.friedman@tcu.edu}}

\author[N. Gadish]{Nir Gadish}
\address{Department of Mathematics, University of Pennsylvania, USA}
\email{\href{mailto:ngadish@math.upenn.edu}{ngadish@math.upenn.edu}}

\author[R. Koytcheff]{Robin Koytcheff}
\address{Department of Mathematics, University of Louisiana at Lafayette, USA}
\email{\href{mailto:koytcheff@louisiana.edu}{koytcheff@louisiana.edu}}

\author[D. Sinha]{Dev Sinha}
\address{Department of Mathematics, University of Oregon, USA}
\email{\href{mailto:dps@uoregon.edu}{dps@uoregon.edu}}

\author[B. Walter]{Ben Walter}
\address{Department of Mathematical Sciences, University of the Virgin Islands, USA}
\email{\href{mailto:benjamin.walter@uvi.edu}{benjamin.walter@uvi.edu}}

\begin{document}

\begin{abstract} 
    We develop bar cohomology of link complements as an invariant of links in homology spheres. In this setting, bar cohomology is a Hopf algebra which is calculable using surfaces and their intersection curves in a link complement. In this first in a sequence of works, we introduce the invariant and show that it defines a canonical subspace of the tensor Hopf algebra, which already encodes information about Milnor's link invariants and provides geometrically significant information beyond them. 

\end{abstract}

\maketitle

\vspace{-3ex}
\tableofcontents


\keywords{\textbf{Keywords:} link, link group, Milnor invariants, bar construction, Hopf algebra, geometric cohomology

\section{Introduction}



It has been understood for some time that the higher linking numbers  collectively known as Milnor's $\mu$-invariants of classical links in the $3$-sphere can be interpreted through cochains \cite{Po80,Tu76} and, more conceptually, in terms of intersections between curves and surfaces in a link's exterior \cite{Ha85, Ha85e,Co90,  penna_higher_2002, Mellor-Melvin, hebda_approach_2012, Hsieh-Kauffman-Tsau, DNOP20}. 
Hain introduced the use of Chen integrals, based on the bar construction (recalled in \cref{S: bar cohomology}) to equate these linking numbers with counts of ordered sequences of intersections along curves. 

We systematize this relationship by treating the entirety of the bar construction, along with its natural algebraic structures, as a link invariant.  In a few senses, the bar construction thus gives a categorification of Milnor invariants.  We use this framework to find invariants beyond Milnor's invariants, which moreover are geometrically computable -- see \cref{ex:3-chain-with-4th}.

This paper is the first in a planned collection of papers developing the bar construction on  cochains of
a link complement as a link invariant. 
Here we introduce the characters, extract a primary invariant, and relate this primary invariant with Milnor's classical higher linking numbers.  We also show that ours is a proper generalization, distinguishing links for which higher linking numbers are not defined. 
In one sequel, almost complete, we use deformation theory of bialgebras to extract finer invariants. One setting where our theory applies coincides with that for the total triple linking invariant of Davis--Nagel--Orson--Powell \cite{DNOP20}, but our theory immediately applies more broadly.
In another forthcoming sequel, we  prove that the bar construction is in fact a concordance invariant, which lifts to a nontrivial map from the space of links to a classifying space of automorphisms of a class of Hopf algebras, factoring through the block embedding space of link concordances. 
 We give an overview of future planned work in Section \ref{subsequent work} below, the common thread being
 the development of explicit geometry of invariants 
 governed by the bar construction, as realized through intersections of surfaces and cobounding curves in the link's complement.

\subsection{Link isomorphism invariance of the bar cohomology}

While our primary focus in this paper will be on links of circles in the $3$-dimensional sphere, the results we state and prove apply far more broadly -- to arbitrary embeddings of oriented codimension-two submanifolds in homology $d$-spheres. Readers interested in high-dimensional linking phenomena, e.g.\ linked surfaces in a $4$-sphere, will find that the theory we develop applies in those cases with virtually no change.

\begin{definition}
Let $nS^1 = \{1,\dots,n\}\times S^1 = \coprod_{k=1}^n S^1$. A {\bf link} in a 3-manifold $M$ is a smooth\footnote{More generally, continuous and locally-flat embeddings.} embedding of circles $L:n S^1 \into M$.  We fix an orientation on $S^1$, thus orienting each component of $L$.
A {\bf pointed link} is  an embedding of  $(n S^1)_+$ obtained from the circles by adjoining a disjoint basepoint.  

More generally, one may take $M$ to be a $d$-manifold and consider embeddings $L:N\into M$, where $N$ is a closed oriented $(d-2)$-manifold with $n$ connected components. 

For any embedding $L: N \into M$, we may write the complement of the image as $M \setminus L$.
\end{definition} 

The constructions we discuss below are invariant under the following equivalence relation. In later work, we will prove that they are concordance invariants as well.
\begin{definition}
    A {\bf link isomorphism} between two (pointed or unpointed) links $L,L'$ in $M$ is a an orientation-preserving diffeomorphism $F:M\to M$ such that $F\circ L = L'$. Two links are said to be {\bf link isomorphic} if there exists a link isomorphism between them.
\end{definition} 
By the isotopy extension theorem, links that are smoothly isotopic are link-isomorphic in this sense, so this equivalence relation is at least as coarse as smooth isotopy.  The two relations agree for $M=S^3$. Isomorphism classes of pointed links are independent of the choices of basepoints and the parameterization of the links. They do, however, depend on the ordering and orientation of the components. 

Our main link invariant is the degree-zero bar cohomology of the link complement, recalled in \cref{S: bar cohomology}, which is a Hopf algebra whose elements are represented by tensor products of cochains. In our setting,  cochains can be given by Seifert surfaces and other hypersurfaces obtained through iterated processes of intersection and cobounding. We also mark our Hopf algebra by the canonical (co)homology classes arising from Alexander duality. Isomorphic links will be shown to determine isomorphic marked algebras, and so we think of these isomorphism types as invariants. Moreover, these invariants can be computed using one's preferred cochain model, e.g., through inductively intersecting surfaces and cobounding their intersection curves, starting with Seifert surfaces as discussed in \cref{S:examples}. 
However, it is not always clear when two such Hopf algebras are isomorphic. We address this challenging flexibility by extracting rigid invariants -- objects which can be compared directly with no automorphisms to account for. 
Our development of these invariants builds upon work on fundamental groups by the second author \cite{Gad23} and by three of the authors with Ozbek \cite{gadish_infinitesimal_2024}, specializing to the setting of a link complement.

Let us describe the markings our Hopf-algebraic invariants possess. All rings in this work are assumed to have a unit.
\begin{definition}\label{def:marked}
Recall that the set of primitive elements of a Hopf algebra $H$ is $\mathrm{Prim}(H) := \{ h\in H \mid \Delta(h) = h\otimes 1 + 1\otimes h \}$.  
When $H$ is a Hopf algebra over some ring $R$, a {\bf marking} of $H$ is a choice of basis of primitive elements for $H$ as a free $R$-module -- that is, a fixed isomorphism $\mathrm{Prim}(H) \cong R^n$, in which case we say that it is of {\bf rank $n$}. Let  $\mathsf{Hopf}_n(R)$ denote the category of connected, commutative, and conilpotent Hopf algebras $H$ together with a marking of rank $n$, with homomorphisms that commute with the given markings.
\end{definition}

For a pointed topological space $X$, let $C^*(X;R)$ denote the functor of singular cochains with coefficients in a commutative ring $R$, augmented by the restriction to the basepoint. For an augmented differential graded algebra $C^*$, let $H^*_{Bar}(C^*)$ denote the cohomology of its bar construction, recalled in \cref{def:bar complex}.
\begin{theorem}[Bar invariant with PID coefficients]\label{T: concordance invariants}
    Fix a principal ideal domain $R$, an $R$-homology $d$-sphere $M$, and an $R$-oriented $(d-2)$-manifold $N$ consisting of $n$ connected components. For example, take $N=nS^1$ and $M=S^3$. Then for every link $L:N\into M$ and for any choice of exterior basepoint, $H^0_{\BBar}(C^*(M\setminus L;R))$ is a marked Hopf algebra such that if $L'$ is a second link isomorphic to $L$ then
    $$H^0_{\BBar}(C^*(M\setminus L); R) \cong H^0_{\BBar}(C^*(M\setminus L'; R)),$$ as marked
     Hopf algebras in $\mathsf{Hopf}_n(R)$.
\end{theorem}
This is proved in \cref{S: bar cohomology}. We can then employ any of the numerous rich invariants of Hopf algebras, each of which gives new link invariants.  Moreover,
our preferred approach to calculations has an intersection-theoretic flavor. Contrast this with Milnor's classical approach to higher linking, using words in nilpotent quotients of the fundamental group, 
which are notoriously difficult to compute and for which geometric interpretation requires substantial labor even in a first case after linking numbers \cite{Mellor-Melvin}.

\subsection{Leading-term invariants}
The first link invariant we extract from the isomorphism type of $H^0_{\BBar}(-;R)$ is its associated graded with 
respect to the natural weight filtration, see \cref{def:weight filtration}\footnote{Note that we use the term \emph{weight} in the sense of Loday--Vallette \cite{Loday-Vallette} and not in the Hodge-theoretic sense used by Hain \cite{hain_mixed_1986}.}. We show that for link complements this associated graded admits a natural embedding into a cofree graded Hopf algebra, yielding a rigid invariant.

\begin{definition}
    Let $R$ be a commutative ring.  Given a generating set $S_1, \cdots S_n$, 
    define the cofree graded Hopf algebra $R\langle S_1,\ldots,S_n\rangle$ as follows.  

Let $V\cong R^n$ be the free $R$-module generated by the set of formal symbols $\{S_1,\ldots ,S_n\}$; then additively  $R\langle S_1,\ldots,S_n\rangle \cong \bigoplus_{p\geq 0}V^{\otimes p}$.
We use bars to denote tensors, so $[S_{i_1} | \cdots | S_{i_p} ] \in V^{\otimes p}.$ 

Its coproduct is defined  by deconcatenation: 
 \[
    \Delta([S_{i_1}|S_{i_2}|\cdots| S_{i_p}]) = \sum_{k=0}^p [S_{i_1}| 
\cdots |S_{i_k}] \otimes  [S_{i_{k+1}}| \cdots| S_{i_p}].
\]

Its (commutative) product is the unsigned shuffle product
    \[
    [S_{i_1}|S_{i_2}|\cdots |S_{i_p}]*[S_{i_{p+1}}|\cdots |S_{i_{p+q}}] = \sum_{\sigma\in \mathrm{Shuffles}(p,q)}[S_{i_{\sigma(1)}}|S_{i_{\sigma(2)}}|\cdots |S_{i_{\sigma(p+q)}}].
    \]

\end{definition}

In fact, $R\langle S_1,\ldots,S_n\rangle$ is cofree as a coalgebra, cogenerated by the projection onto its summand $V$.

\begin{theorem}[Leading-term bar invariant] \label{intro-thm:primary invariant}
    Fix a commutative Noetherian ring $R$, an $R$-homology $d$-sphere $M$, and an $R$-oriented $(d-2)$-manifold $N$ consisting of $n$ connected components.  For example, take $N=nS^1$ and $M=S^3$. For every link $L:N\into M$ and any choice of external basepoint, the associated graded of $H^0_{\BBar}(M \setminus L)$ with respect to tensor weight, denoted $\Primary(L)$, is equipped with a natural embedding of graded algebras $i_L:\Primary(L)\into R\langle S_1,\ldots,S_n\rangle$, 

    The image $i_L(\Primary(L))\subseteq R\langle S_1,\ldots,S_n\rangle$ is a link isomorphism invariant, meaning that if there exists a link isomorphism $L\sim L'$, then $i_{L}(\Primary(L))=i_{L'}(\Primary(L'))$ as submodules of $R\langle S_1,\ldots,S_n\rangle$. In particular, the invariant is independent of basepoints.
    
    When $R$ is a principal ideal domain, $\Primary(L)$ is a Hopf algebra, and the embedding is a map of Hopf algebras.
\end{theorem}

We will show in Section \ref{S:examples} that this theorem is geometric in the sense that $\Primary(L)$ is computable directly through intersecting and cobounding  proper hypersurfaces in $M\setminus L$. The remaining statements are proved in  \cref{prop:H0Bar is Hopf} in \cref{S:primary}.}

    One can query the subspace $i_L$ 
    to define concrete invariants.  For example, for any given element $P\in R\langle S_1, \cdots S_n\rangle$ the question of whether this element lies in the image of $i_L$ constitutes a binary link isomorphism invariant, which we compare with Milnor's $\bar\mu$-invariants below.  

\begin{terminology}
    We call the embedding $i_L:\Primary(L)\into R\langle S_1,\ldots,S_n\rangle$ the  {\bf  leading-term Milnor--Hopf algebraic invariant of a link} $L$, or simply the {\bf  leading-term bar invariant}. Often, we abuse notation and write $\Primary(L)$ for the embedding or for its image in $R\langle S_1,\ldots,S_n\rangle$.

    The formal symbols $S_1,\ldots,S_n$ are identified with a basis of the link complement's first cohomology,
    with $S_i\in H^1(M\setminus L;R)$  the class that measures linking of $1$-cycles with the $i$-th link component.  In three-dimensional geometric cohomology, this class
    is represented by a Seifert suface for the $i$-th component. 
    
\end{terminology}

We name this invariant after Hopf and Milnor because the Hopf algebra $\Primary(L)$ reflects Milnor's $\bar\mu$-invariants, as discussed next, while the bar construction can be used to define all higher rational Hopf invariants \cite{Sinha-Walter:2013}. 


\subsubsection{Relationship with the fundamental group}\label{relationshippi1}

The leading-term invariant $\Primary(L)$ is related to the fundamental group of the complement of $L$ as follows. For a pointed link $L$ in $M$, let $\pi(L)$ denote the fundamental group $\pi_1(M\setminus L)$ of the link complement $M\setminus L$, based at the chosen basepoint. 
For any ring $R$, the group ring $R[\pi(L)]$ is a Hopf algebra with respect to the usual convolution product and the coproduct determined on elements of the group by $\Delta(g) = g \otimes g$. The augmentation ideal is defined as $I_R(L): = (g-1\mid g\in \pi(L) ) \lhd R[\pi(L)]$.
\begin{proposition}\label{P: associated graded augmentation}
    Let $R$ be a field and let $L_0,L_1$ be two pointed $n$-component links in an $R$-homology $d$-sphere $M$. Then $\Primary(L_0) = \Primary(L_1)$ if and only if there exists an isomorphism of associated graded Hopf algebras
    \[
    \bigoplus_{p\geq 0} I_R(L_0)^p/ I_R(L_0)^{p+1} \cong \bigoplus_{p\geq 0} I_R(L_1)^p/ I_R(L_1)^{p+1}
    \]
    that furthermore preserves the classes of meridians. The latter condition means that if $\mathpzc{m}_k(L_0)$ and $\mathpzc{m}_k(L_1)$ are any pointed meridians for the $k$-th link component of their respective links, then the isomorphism must send $\mathpzc{m}_k(L_0)$ to $\mathpzc{m}_k(L_1)$ modulo $I_R(L_1)^2$.
\end{proposition}
The proof of \cref{P: associated graded augmentation} is at the end of \cref{S:primary}.

\begin{remark}[Low-degree truncations]\label{S:bounded truncations}
The leading-term invariant $\Primary(L)$ breaks up into graded components and can thus be considered as a sequence of invariants of growing complexity: $\Primary(L)_2$ is sensitive to pairwise linking numbers, $\Primary(L)_3$ detects triple linking, and $\Primary(L)_4,\Primary(L)_5,\ldots$ extend this further. Equality of $\Primary(L)_p$ for $p\leq p_0$ translates to an isomorphism of associated graded truncations of group rings $R[\pi(L)]/I^{p_0+1}$: when $R$ is a field,
\begin{equation}\label{eq:bdd truncations}
\Primary(L)_{\leq p_0}=\Primary(L')_{\leq p_0} \iff \bigoplus_{0\leq p\leq p_0} I_R(L)^p/I_R(L)^{p+1} \cong \bigoplus_{0\leq p\leq p_0} I_R(L')^p/I_R(L')^{p+1}
\end{equation}
through an isomorphism of algebras that preserves the respective meridians.
One can calculate these invariants in degree $p\leq p_0$ without having to compute the entire space $\Primary(L)$. Every such bounded calculation involves finitary $R$-linear algebra.
\end{remark}

\subsubsection{Relation to Milnor's invariants}
We also think of $\Primary(L)$ as a refinement of Milnor's $\bar\mu$-invariants of links, as indicated by the following result. Throughout, for ease of notation we write $\Z_\Lambda$ for the ring $\Z/\Lambda \Z$ where $\Lambda\in \Z$, and $\Z_0=\Z$.
\begin{theorem}[Refinement of $\bar\mu$-invariants]
    \label{T:Milnor refinement}
    Fix a coefficient ring $R=\Z_\Lambda$ for some $\Lambda\in \Z$ (possibly zero) and let $L_0$ and $L_1$ be two $n$-component links in $S^3$. If $\Primary(L_0)=\Primary(L_1)$, then for every $p\geq 2$, the $\bar\mu$-invariants $\bar \mu_{i_0\cdots i_r}$ of $L_0$ are well-defined and vanish in $R$ for all multi-indices $(i_0,\ldots,i_r)$ with $1\leq i_0,\ldots i_r\leq n$ and $r< p$ if and only if the same invariants of $L_1$ are well-defined and vanish in $R$. Moreover, if $\Lambda \neq 0$ and these invariants vanish for a given $p$, then for any multi-index $(i_0,\ldots,i_p)$ with $1\leq i_0,\ldots,i_p\leq n$, we have $\bar \mu_{i_0\cdots i_p}(L_0)$ is equal to $\bar\mu_{i_0\cdots i_p}(L_1)$ up to multiplication by a unit in $R^\times$.
\end{theorem}

This theorem provides one sense in which bar cohomology categorifies Milnor's invariants, with these classical numerical invariants encoded by evaluating cohomology on a homology class.   Our encoding is through a canonical embedding of a submodule, allowing linear combinations and making certain identities between Milnor invariants into formal consequences of structures on this submodule.  \cref{T:Milnor refinement} is proved as parts of stronger results, namely \cref{prop:comparison with milnor} and \cref{C:Milnor} in \cref{S: Milnor}. The relevant definitions and a more detailed account of the relationship between $\bar\mu$-invariants and $\Primary(L)$ are given in that section. In particular, \cref{C:Milnor} demonstrates a use of $\Primary(L)$ to determine $\bar \mu$-invariants of $L$ up to units in the rings in which they are defined. 

\subsubsection{Beyond Milnor invariants}
The leading-term invariant $\Primary(L)$ gives additional information that is invisible to $\bar\mu$-invariants. It thus generalizes Milnor's invariants. 
It does so in a sense by allowing higher invariants to be defined not only given vanishing of lower ones but also given equality or more generally vanishing linear combinations. 

\begin{figure}
    \includegraphics[scale=0.55]{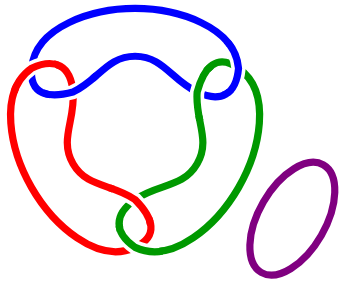} 
    \hspace{2cm}
    \includegraphics[scale=0.55]{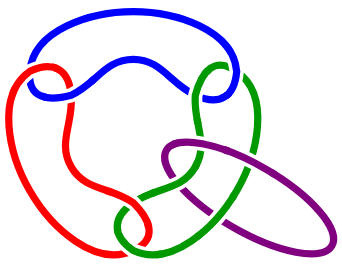} 
\caption{Two 4-component links distinguished by our leading-term invariant $\Primary(L)$}
\label{F:two-4-comp-3-chains}
\end{figure}

\begin{example}\label{ex:intro-primary} 
Consider the two $3$-chain links together with a fourth component shown in Figure \cref{F:two-4-comp-3-chains}.  Order the components so that $1\leftrightarrow$ \blue{blue}, $2\leftrightarrow$ \red{red}, $3\leftrightarrow$ \green{green}, and $4\leftrightarrow$ \purple{purple}.
    For the link $L$ on the left, the leading-term invariant $\Primary(L)$ contains $[S_1|S_2|S_4]+[S_2|S_3|S_4]+ [S_3|S_1|S_4]$, while for the link $L'$ on the right, $\Primary(L')$ does not contain this element. It follows that the two links are not isomorphic.  Since the pairwise linking numbers of the $3$-chain are all $1$, no (classical) triple linking numbers can be defined for them; likewise, no higher-order $\bar\mu$-invariants can be defined for any sublink containing more than one of the three components in the 3-chain.  But in both $L$ and $L'$, the fourth component forms an unlink with any other component, so the (classical) Milnor invariants cannot distinguish these links.
    On the other hand, the more recently discovered total triple linking invariant of Davis, Nagel, Orson, and Powell \cite{DNOP20} does distinguish these links.
    
    The calculation of $\Primary(L)$ is explicit and geometric: pick planar Seifert surfaces for all components, denoted $\Sigma_i$ for $1\leq i\leq 4$ and representing the respective cohomology classes $S_i$, and observe that $(\Sigma_1\cap \Sigma_2) + (\Sigma_2\cap \Sigma_3) + (\Sigma_3\cap \Sigma_1) = \partial \Sigma_{123}$ is the boundary of a surface spanning the hole in the middle of the $3$-chain (see \cref{fig:intro-surfaces}). The link on the right has the purple component piercing the secondary surface $\Sigma_{123}$ once, thus excluding the polynomial $([S_1|S_2]+[S_2|S_3]+[S_3|S_1])|[S_4]$ from the leading-term invariant. See \cref{ex:3-chain-with-4th} for further details, and see \cref{ex:hopf-linked-3-chains,ex:6-comp-sum-Wh-double-Borr} for more elaborate links that showcase the strengths of $\Primary(L)$.
    \begin{figure}[h!]
        \centering
        \includegraphics[scale=0.45]{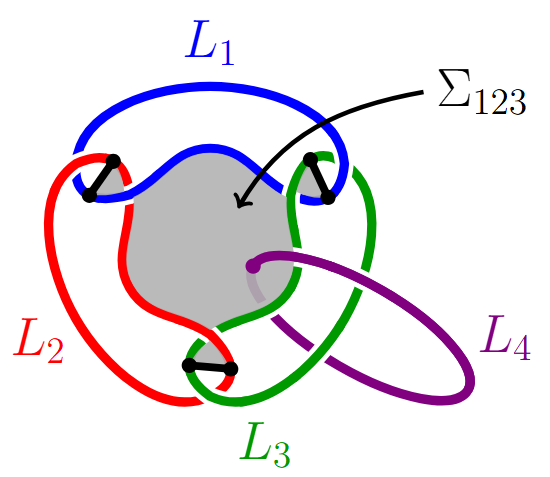}

        \caption{A secondary surface cobounding the intersections of the Seifert surfaces of the $3$-chain.}
        \label{fig:intro-surfaces}
    \end{figure}
\end{example}

\begin{remark}
For applications as well as comparison with Milnor's $\bar \mu$-invariants, one would hope to define the invariant $\Primary(L)$ as a Hopf algebra over rings of the form $R=\Z_\Lambda$ for arbitrary $\Lambda\in \Z$. However, when the coefficient ring is not a PID, the leading-term invariant $i_L(\Primary(L))\subseteq R\langle S_1,\ldots,S_n\rangle$ does not form a sub-coalgebra -- see \cref{ex:non-coalgebra}. Nonetheless, the subspace $i_L(\Primary(L))\subseteq R\langle S_1,\ldots,S_n\rangle$ is well-defined, closed under the shuffle product, and provides information about the $\bar\mu$-invariants of $L$.

\end{remark}

\subsection{Alternate perspectives and future work}\label{subsequent work}

One can also frame our approach through the notion of polynomial function on a group, originally due to Passi \cite{passi_polynomial_1968}.
\begin{definition}
    Let $G$ be a group, $R$ a ring, and $I$ the augmentation ideal of $R[G]$. A function $f:G\to R$ is {\bf polynomial of degree $n\in \N$} if its $R$-linear extension $R[G]\to R$ annihilates $I^{n+1}$.
\end{definition}

For example, group homomorphisms to $R$ with its additive structure
are linear functions, that is, polynomial of degree $1$.
The ring of polynomial functions on $G$ forms a commutative conilpotent Hopf algebra under $(\Delta f)(x\otimes y) := f(xy)$ and $(f*g)(x) := f(x)g(x)$. There are  equivalent notions of polynomiality, involving discrete derivatives\footnote{In particular, the ``derivative of a function $f : G \to R$ in the direction of $a \in G$'' is defined as $\Delta_af (g) = f(ga) - f(g)$.} or conilpotence.

Key examples of polynomial functions are coefficients
of Magnus' expansion for free groups, which Milnor used to define his link invariants. Milnor's original definition is thus the evaluation of polynomial functions on longitudes of components. Substantial
difficulties arise both because the link groups
in question are not free and 
because longitudes are not naturally based.
One must therefore restrict to evaluations on polynomial functions that are well-defined on the link group while also being constant on conjugacy classes, which requires vanishing evaluation of lower-order functions.

With the applications of the current paper in mind, the second author shows in \cite{Gad23} that degree zero bar cohomology produces all polynomial functions on any fundamental group, giving a concrete generalization
of Magnus expansion for presented groups through ``letter braiding.'' In the present work we focus on bar cohomology itself as a link invariants.  Moreover, the bar cohomology perspective leads
to a canonical system of ``coordinates'' up to first order, defined
though intersections with Seifert surfaces.

\begin{definition}

    Fix ${\mathcal S} = \{S_1, \cdots, S_m\}$, a finite collection of oriented immersed surfaces in the smooth manifold $N$, and let $\Omega^{\mathcal S} N$
    be the space of immersed smooth based curves in $N$ whose intersections with all $S_i$ are transverse.

    For a geometric cochain bar monomial $b = [S_1 | \cdots | S_n]$, define the {\bf intersection function}, whose value on a curve $c  \in \Omega^{\mathcal S} N$ is 
    \[
    \iota_b(c) = \sum_{\{ t_1 < t_2 < \cdots < t_n \; | \; c(t_j) \in S_{j} \}} \operatorname{sgn}(c(t_1)\pitchfork S_1)\cdots \operatorname{sgn}(c(t_n)\pitchfork S_n)
    \]
    where $\operatorname{sgn}(c(t_j)\pitchfork S_j)$ is the sign of the intersection of the curve and the surface at the point $c(t_j)$.

\end{definition}

For example, $\iota_{[S | T | S]}$  is the signed count of the number of times a curve crosses through $S$ and then later through $T$ and then through $S$ again. These functions were considered by Hain \cite{Ha85} to give a geometric interpretation of Milnor's invariants (see \cite{hebda_approach_2012,penna_higher_2002} for this perspective). 
In general, some linear combinations of these surface-intersection counts descend to well-defined functions on the link group $G = \pi_1\left(M \setminus L, *\right)$, in which case they are polynomial. In future work, we will show that the resulting functions on the fundamental group are equivalent to those calculated through the geometric cochain bar complex defined in Appendix~\ref{A:geometric}, and by comparison with
singular cohomology give all polynomial invariants.

Drawing from  these perspectives and the ones developed
in this paper, we see many possible applications and extensions of this theory, starting with two forthcoming 
papers mentioned above.

\begin{itemize}
    \item The leading term invariant developed here captures the associated graded of the bar construction as a canonically embedded object.  Equivalently by \cref{P: associated graded augmentation} it captures the associated graded of the augmentation ideal filtration of the fundamental group ring. Classification of bialgebras with the same associated graded is captured by a form of deformation theory developed by Gerstenhaber and Schack \cite{Gerstenhaber-Schack}.  We use this to define concrete geometric link invariants, which capture 
    the ``two step associated graded'' of the augmentation
    ideal filtration of the group ring of the fundamental group.\\

    \item Invariants in geometric topology which are organized around the lower central series filtration of the fundamental group are often concordance invariants.  
    In forthcoming work we find this is the case for our bar cohomology invariants, which are naturally organized around the augmentation ideal of the group ring, as well.  More substantially, we lift this concordance invariance 
    to give maps from  spaces
    of block embeddings for a large class of spaces
    of links to classifying spaces 
    based on the Hopf algebras which define our theory.\\

    \item In work in progress, the second author has classified the rings of polynomial functions of links for which all pairwise linking numbers are one,
    and thus Milnor invariants give 
    no further information.  This work illustrates the more generic phenomena
    which arise when most or all Milnor invariants
    do not vanish, in which case one can make use of their linear
    dependence. \\ 

    \item Starting from a link diagram and
    using ``mostly flat'' Seifert surfaces, these invariants can be calculated purely combinatorially.
    This will lead to  programmable invariants, which should be practically useful once algorithms which make use of structure in 
    the bar construction are implemented.\\

    \item While one subsequent paper explores secondary invariants arising from the Gerstenhaber-Shack cohomology, we have also started to develop tertiary and higher extension invariants, comparing  filtered Hopf algebra structures past their first-order deformations.  \\

    \item Polynomial functions defined by bar cocycles, through intersection with surfaces, could extend the reach of Milnor's classical approach. For example, from this perspective, to grasp the difference between the two links in \cref{ex:3-chain-with-4th} one can use the ``linking number with the length-three chain formed by the first three components,'' which is the evaluation of a quadratic function. We conjecture that evaluation of degree $n$ polynomial functions on longitudes could be equivalent to knowledge of the degree $n+1$ leading term invariant, but based on this example expect easier application of longitude evaluation in some settings.\\

    \item Another planned future direction is to study in more detail the setting of links in arbitrary 3-manifolds, potentially connecting to recent work of Kuzbary \cite{Kuzbary:AGT}
and Stees \cite{stees_milnors_2023}.  There we already
have a conjectural version of the leading-term invariant.\\

\item There is potential for refinement by replacing
the bar construction, which only requires associativity,
with constructions which involve homotopy commutativity.  
There should also be ``savings'' as one can quotient by
tautological classes, as the Harrison complex does
in characteristic zero.\\

    \item Concordance invariants of links have been studied extensively using surfaces in a distinct way,
    namely through gropes, Whitney towers, and similar constructions; see \cite{schneiderman2005whitney,conant2012whitney,kim2015whitney}. Whitney towers study link groups through their nested commutators, in contrast with the bar construction which does this via the augmentation ideal filtration. While there are distinctions in group theory between the lower central series and the dimension series, especially with
    varied coefficients, there are comparisons to be made.  A tantalizing question is whether there
    are comparisons to be made  between the two central series through surface-level 
    constructions.
\end{itemize}


\subsection{Conventions and models}

Throughout this paper, $R$ is a unital commutative Noetherian ring. In various contexts, we get stronger results by restricting $R$ further to be a PID or field,
and we will always be explicit about such restrictions.

All tensor products are  taken over $R$. Given an $R$-module $M$, its tensor coaglebra is $T^c(M) = \bigoplus_{p\geq 0} M^{\otimes p}$. The reader will note that this looks identical to the tensor algebra on $M$, except that here it is equipped with a deconcatenation coproduct rather than a concatenation product; see equation \eqref{eq:tensor coproduct} below.  
All the coalgebras we will consider are counital and conilpotent. In the category of such coalgebras, $T^c(M)$ is the cofree coalgebra on $M$, meaning that the obvious projection $pr:T^c(M)\to M$ defines a bijection
\[
\Hom_{\mathrm{coAlg}}(C,T^c(M)) \cong \Hom_{R}(C,M).
\]
For an introduction to coalgebras, see \cite[Section 1.2]{Loday-Vallette}.

The algebraic input to our machinery is cochain algebras with coefficients in $R$.
A {\bf differential graded algebra} over $R$, or {\bf dg $R$-algebra} for short, is a bounded-below, graded, associative, unital $R$-algebra $A^\bullet$ equipped with a square-zero derivation of degree $+1$.  In a graded setting, some authors use the term ``(co)commutative'' to mean ``graded-(co)commutative''; since we will focus on elements of degree 0, we use these terms to mean ``strictly (co)commutative.''

The bar construction is a central character in our story. 
See \cref{S: bar cohomology} for this construction on associative algebras, and \cref{A:geometric} for homotopy coherent generalizations. This construction is robust in the sense that it is amenable to using varied cochain models for link complements, e.g., singular or smooth cochains, differential forms, or the geometric cochains \cite{FMS-foundations} discussed in the next paragraph. For cochain models related by a quasi-isomorphism of associative or $A_\infty$-algebras, there are quasi-isomorphisms of bar constructions and thus the resulting invariants agree. For example, the classical de Rham morphism from differential forms to smooth cochains is not a homomorphism of associative algebras, but Gugenheim \cite{Gu76} showed (in slightly different language) that it extends to an $A_\infty$-quasi-isomorphism, which induces a quasi-isomorphism of bar complexes. In this case, our resulting invariants therefore coincide for different cochain models.

For performing computations geometrically -- using intersections of surfaces -- our preferred cochain model is geometric cochains $C^*_\Gamma(-)$ \cite{FMS-foundations}. Geometric cochains possess only a Leinster partial algebra structure (see \cite{McC06} for definitions and for the PL-chains analog), but this is sufficient to obtain an analog of the bar complex developed by the first author \cite{GBF47}, described in \cref{A:geometric}. The question remains whether this geometric bar construction is quasi-isomorphic to the one obtained through singular cochains. We recently learned about forthcoming dissertation work of Pizzi \cite{Pizzi} in which he constructs such a natural quasi-isomorphism, so the two bar cohomologies should indeed coincide.

Our proofs will generally be written in terms of an arbitrary cochain model $C^*(X;R)$ with a globally defined product, but we will often frame examples in the language of geometric cochains, illustrating that the two machineries (geometric cochains and bar cohomology invariants) combine to generate invariants that are visually computable. When computing our invariants from geometric cochains, the reader is asked to interpret the bar construction as the one in the sense of \cref{A:geometric}. A moment's reflection shows that our arguments go through with minimal changes.

\subsection{Acknowledgments}
This project has been gestating for quite some time and has benefited from many useful conversations and partial collaborations. In particular, we acknowledge the profound contributions of Aydin Ozbek, who tragically passed away during the writing of this work. We owe a debt of gratitude to Anna Cepek and Ziyal Jandrasi, who handed this project off to us in its infancy. We thank Andres Fernandez Herrero for countless useful conversations about Hopf algebras and beyond. The project benefited from the ICERM workshop ``Links in Dimensions 3 and 4,'' where meaningful exchanges of ideas took place. 
Lastly, this work benefited from conversations with Google Gemini, which assisted us in finding references and catching small algebraic mistakes.

RK  was supported by the Louisiana Board of Regents Support Fund, contract number LEQSF (2019-22)-RD-A-22 and by the National Science Foundation, Award No.~DMS-2405370.
GF was partially supported by  grant \#839707 from the Simons Foundation.  This work was initiated by a visit of DS by NG in 2022, funded by the University of Oregon and NSF DMS No. 2039316.

\section{Hopf algebraic invariants via the bar construction}
In this section we set up the relevant terminology and establish the invariance of the associated graded of the bar cohomology, which will be our primary, ``leading-term,'' invariant. In later sections we will extract further link invariants from it and establish isomorphism invariance.

\subsection{Bar cohomology}\label{S: bar cohomology}

Recall that a {\bf differential graded  $R$-algebra $A$} (or {\bf dg-$R$-algebra} for short) is a bounded-below cochain complex of $R$-modules $(\cdots \to A^i\to A^{i+1}\to \cdots )$ with $A^i=0$ for $i<N$ for some $N$, equipped with a graded map $A\otimes_R A\to A$ making $A$ into a unital associative algebra and satisfying the Leibniz rule with respect to the differential. An {\bf augmentation} of $A$ is an $R$-algebra homomorphism $\epsilon:A\to R$, splitting the unit: $A= R\oplus \bar A$ with the $R$ summand in degree $0$. The kernel $\bar A = \ker \epsilon$ is called the {\bf augmentation ideal}.
\begin{definition}
\label{def:bar complex}
    The {\bf bar construction} of an augmented differential graded associative $R$-algebra $A$ with augmentation ideal $\bar A$ is the totalization of the second quadrant bicomplex\footnote{Our definition here coincides with the two-sided bar complex $\BBar(R,A,R)$ using the $A$-module structure on $R$ induced by the augmentation map.  The equivalence uses the fact that $\bar A^{\otimes p} \cong R \otimes_R \bar A^{\otimes p}\otimes_R R$, and, for the boundary formula, the fact that $\bar A$ acts by zero on $R$. When $A$ is flat over $R$, the cohomology of this complex computes $\operatorname{Tor}_A^*(R,R)$.}
    \[
    \BBar(A) := \bigoplus_{p\geq 0} \bar A^{\otimes p}
    \]
    in which the $(-p,q)$-bigraded part is the summand $(\bar A^{\otimes p})_q$ of cohomological grading $q$ \cite{EM53,McC}. Elementary tensors $a_1\otimes \cdots \otimes a_p$ will be denoted with brackets and with the tensor product symbols replaced by bars: $[a_1|\cdots|a_p]$. By convention $\bar A^{\otimes 0}=R$, with the empty tensor identified with $1\in R$.
    
    The vertical and horizontal differentials of the bicomplex are
    \[
    d_v[a_1|\ldots|a_p] := \sum_{i=1}^p (-1)^{|a_1|+\ldots + |a_{i-1}|-(i-1)} [a_1|\cdots | da_i | \cdots | a_p]
    \]
    \[
    d_h[a_1|\cdots|a_p] := \sum_{i=1}^{p-1} (-1)^{|a_1|+\ldots+|a_i|-i}[a_1|\cdots |a_ia_{i+1}|\cdots |a_p].
    \]
    In particular, when all $a_i$'s have cohomological degree $1$, the differentials involve no signs.

    Denote $\BBar(A) := \operatorname{Tot} \left(\bigoplus \bar A^{\otimes p} ,d_v,d_h\right)$, the totalization, which is equipped with the total differential $d_{\BBar}:=d_v+d_h$.  Then denote $H^q_{\BBar}(A) := H^q(\BBar(A), d_{\BBar})$.
    In particular, when $\bar A^{q} = 0$ in nonpositive degrees $q\leq 0$, we have $Bar^0(A) = \bigoplus_{p\geq 0}(A^1)^{\otimes p}$.

    For a pointed topological space $X$ and a coefficient ring $R$, denote $\BBar(X;R) := \BBar(C^*(X;R))$, the bar construction of the singular cochain algebra of $X$ with coefficients in $R$, augmented by restriction to the basepoint of $X$. Similarly, denote $H^q_{\BBar}(X;R) := H^q_{\BBar}(C^*(X;R))$.
\end{definition}

 A geometric model, using geometric cochains instead of singular ones, as discussed in \cref{A:geometric}, requires a variant of the bar construction that we discuss in that appendix.

The bar construction has a 
deconcatenation coproduct
\begin{equation}\label{eq:tensor coproduct}
\Delta [a_1|\cdots|a_p] := \sum_{i=0}^p [a_1|\cdots | a_i]\otimes [a_{i+1}|\cdots |a_p].
\end{equation}
When $i=0$ or $p$, the empty monomial is  $1\in R$.
The coproduct is strictly coassociative and satisfies the Leibniz rule with respect to the total differential $d_{\BBar}$. 

It is immediate from definitions that the bar construction
is functorial for maps of dg $R$-algebras.

When working over a field, the K\"unneth formula implies that $H^*_{\BBar}(A)$ is a coalgebra under the induced coproduct
\[
H^*(\BBar(A))\xrightarrow{H^*(\Delta)} H^*(\BBar(A)\otimes_R \BBar(A)) \cong H^*(\BBar(A))\otimes_R H^*(\BBar(A)).
\]
In this work we are interested in the bar construction defined over more general commutative rings, such as $\Z_m$, so we need to work harder to establish a similar algebraic structure. In particular, we show below in \cref{prop:primary over nonPID} that for link complements, an associated graded of $H^0_{Bar}(X;R)$ embeds into the cofree coalgebra.

A key tool in working with the bar construction is its natural filtration by tensor weight, which is
a special case of a standard filtration 
for  augmented coalgebras.

\begin{definition}\label{def:weight filtration}
    An {\bf augmentation}\footnote{In the literature, this is commonly called a {\bf co-augmentation}.} of a coalgebra $C$ is a choice of element $1\in C$ that evaluates to $1\in R$ under the counit. Given an augmentation, define the {\bf reduced coproduct}
    \[
    \barDelta(x) = \Delta(x) - (x\otimes 1 +1\otimes x).
    \]
    An element $x\in C$ is {\bf primitive} if $\barDelta(x)=0$, or equivalently if $\Delta(x)=x\otimes 1 + 1\otimes x$. Denote the set of primitive elements by $\mathrm{Prim}(C)$.
    
    Augmented coalgebras are naturally filtered by vanishing of the reduced coproduct: define iterates of this operation recursively by $\barDelta^1 = \barDelta$ and $\barDelta^{p+1} = (\barDelta\otimes \id^{\otimes (p-1)})\circ \barDelta^{p}$. For every augmented coalgebra $C$, define the {\bf weight filtration}\footnote{Note that we use the term \emph{weight} in the sense of Loday--Vallette \cite{Loday-Vallette} and not in the Hodge-theoretic sense used by Hain \cite{hain_mixed_1986}.} by
    \[
    F_p C := \ker(\barDelta^{p}).
    \]
    We say $C$ is {\bf conilpotent} if $\bigcup_{p\geq 0}F_p C = C$, so every $c\in C$ is annihilated by some finite iterate of the reduced coproduct.
\end{definition}
Note that every homomorphism of augmented coalgebras $C_0\to C_1$ preserves the respective weight filtrations. In particular, if $C$ is a Hopf algebra with product $m:C\otimes C\to C$, then $m$ is a homomorphism of augmented coalgebras, and thus every such product is filtration-preserving. 

For the bicomplex $\BBar(A)$, the weight filtration coincides with tensor degree: $F_p \BBar(A) = \bigoplus_{p'\leq p} \bar A^{\otimes p'}$.  
If we abbreviate $F_p:=F_p \BBar(A)$,  its associated graded has summands $F_p/F_{p-1}\cong \bar A^{\otimes p}$.  Define the {\bf top-weight} or {\bf leading} term of some $a \in \BBar(A)$ of weight $p$ to be 
\[
a \longmapsto pr^{\otimes p}\barDelta^{p-1}(a) \in \bar A^{\otimes p},
\]
where $pr:\BBar(A)\to \bar A$ is the projection onto the tensors of weight 1, i.e., the {\bf cogenerators}. 
Concretely, if $a = a_p+a_{p-1}+\ldots+a_0\in \BBar(A)$ is a (nonhomogeneous) tensor with $a_i\in \bar A^{\otimes i}$ and $a_p\neq 0$, the leading term is $a_p$.

The weight filtration on a differential graded conilpotent coalgebra $C$ gives rise to a second-quadrant spectral sequence called the {\bf weight spectral sequence}.
We also call it the {\bf bar spectral sequence} when $C=\BBar(A)$, in which case it also arises as the spectral sequence of $\BBar(A)$ as a double complex.  

\begin{lemma}[Quasi-isomorphism invariance]\label{lem:quasi-iso invariance} 
    If $h:A_0\to A_1$ is a quasi-isomorphism of augmented dg-$R$-algebras whose underlying $R$-modules are flat, then the induced map $\BBar(h):\BBar(A_0)\to \BBar(A_1)$ is a quasi-isomorphism.
\end{lemma}
\begin{proof}
    First, for $i=0,1$, observe that since $A_i$ is flat over $R$, then its direct summand $\bar A_i$ is also flat. Consider the map induced by $h$ between the two weight spectral sequences. On the $E_0$-page, it is the map induced on tensor powers $h^{\otimes p}:\bar A_0^{\otimes p}\to \bar A_1^{\otimes p}$.

    But since $\bar A_i$ is flat over $R$ and bounded below, the operation $(-)\otimes_R \bar A_i$ preserves quasi-isomorphisms of bounded-below complexes. Indeed, in any cohomological degree the tensor product complex involves only finitely many terms from each complex, so to prove that a map $B\to B'$ induces an isomorphism $H^q(B\otimes _R A_i)\to H^q(B'\otimes _R A_i)$ on $q$-th cohomology we may assume that $A_i$ is bounded above; replacing terms of extremely high cohomological degree by zero has no effect on degree $q$. But a bounded-above complex of flat modules is $K$-flat \cite[\href{https://stacks.math.columbia.edu/tag/06XY}{Section 06XY}]{stacks-project}, meaning that tensoring with it is quasi-isomorphism invariant and in particular it induces an isomorphism on $q$-th cohomology. Since $q$ was arbitrary, $(-)\otimes_R A_i$ preserves quasi-isomorphisms.

    By induction, it follows that the maps $(\bar A_0^{\otimes i-1}\otimes \bar A_0 \otimes \bar A_1^{\otimes j})\to (\bar A_0^{\otimes i-1}\otimes \bar A_1 \otimes \bar A_1^{\otimes j})$ are quasi-isomorphisms, showing that
    \[
    \bar A_0^{\otimes p} \to \ldots \to (\bar A_0^{\otimes i}\otimes \bar A_1^{\otimes j}) \to (\bar A_0^{\otimes i-1}\otimes  \bar A_1^{\otimes j+1}) \to \dots \to \bar A_1^{\otimes p}
    \]
    is a quasi-isomorphism. Therefore, $\BBar(h)$ induces an isomorphism on the $E_1$-pages, and a standard spectral sequence argument shows that it must already have been a quasi-isomorphism \cite[Theorem 5.5.11]{WEIB}.
\end{proof}

\begin{corollary}[Topological invariance]
    The functor $X\mapsto \BBar(X;R)$ is invariant under $R$-cohomology isomorphisms. That is, if $f:X_0\to X_1$ is a pointed continuous map of topological spaces, and $f$ induces an isomorphism in cohomology with coefficients in $R$, then $\BBar(f^*): \BBar(X_1;R)\to \BBar(X_0;R)$ is a quasi-isomorphism.
\end{corollary}
\begin{proof}
    Recall our convention that the ring $R$ is assumed to be Noetherian. For every space $X$ and in every cohomological degree, singular cochains are
    \[
    C^q(X;R) = \Hom( R^{\oplus X_q}, R) = \prod_{\sigma \in X_q} R
    \]
    where $X_q$ is the set of singular $q$-simplices in $X$. By a result of Chase \cite[Theorem 2.1]{Chase-infiniteproducts}, if $R$ is a Noetherian ring, an infinite product of $R$ is a flat $R$-module, so $C^*(X;R)$ is degree-wise flat. 
    
    Given a pointed continuous map $f:X_0\to X_1$, the pullback on $R$-cochains $f^*:C^*(X_1;R)\to C^*(X_0;R)$ is a homomorphism of augmented $R$-algebras. Lemma \ref{lem:quasi-iso invariance} now implies that $f^*$ induces a quasi-isomorphism on the bar constructions.
\end{proof}

We will not recall the definitions of $E_\infty$-algebras here, but the reader should keep in mind that these are the homotopy coherent analogs of commutative algebras of which simplicial and singular cochain algebras are examples. Geometric cochains are an example of a Leinster homotopy commutative algebra (see \cref{A:geometric}), which are equivalent to $E_\infty$-algebras and to which a modified version of the treatment below applies using an appropriate version of the bar construction -- see \cref{A:geometric}. An $E_\infty$-Hopf algebra is an $E_\infty$-algebra in the category of coassociative coalgebras. See the work of Fresse \cite{Fre07} for detailed definitions.
\begin{lemma}
\label{lem:product on bar}
    Let $R$ be a commutative ring, and let $A$ be an $E_\infty$-algebra over $R$ which is a flat $R$-module. Then the dg-coalgebra $\BBar(A)$ is equivalent to an $E_\infty$-dg-Hopf algebra through a quasi-isomorphism of filtered coalgebras. In particular, the cohomology $H^*_{\BBar}(A)$ is a commutative algebra and the product respects the weight filtration.
\end{lemma}
\begin{proof}
    Fresse \cite{Fresse_Hopf} shows that whenever $A$ is a cofibrant $E_\infty$-algebra over an arbitrary ground ring $R$, its bar construction $\BBar(A)$ has an essentially unique functorial $E_\infty$-algebra structure in the category of coassociative coalgebras in chain complexes. This is Fresse's definition of an $E_\infty$-Hopf algebra.  (He does not discuss antipodes of Hopf algebras, but these exist since the bar construction is graded, connected and conilpotent.)  The weight filtration is determined by the vanishing of the reduced coproduct $\barDelta$, and the multiplication $m$ is a coalgebra homomorphism, so the product is a filtered map.
    
    When $A$ is not cofibrant, we may construct a cofibrant replacement $h:P\xrightarrow{\sim} A$ where $P$ is a cofibrant $E_\infty$-algebra and $h$ is a quasi-isomorphism. Now Fresse's result endows $\BBar(P)$ with an $E_{\infty}$-algebra structure.
    But $\BBar(h):\BBar(P)\to \BBar(A)$ is a map of filtered coalgebras, and by \cref{lem:quasi-iso invariance} it is also a quasi-isomorphism. This is the claimed $E_\infty$-structure.

    To see that this structure induces a filtered commutative algebra structure on $H^*(\BBar(P))\cong H^*(\BBar(A))$, let $m:\BBar(P)\otimes \BBar(P)\to \BBar(P)$ be the product that underlies the $E_\infty$-Hopf structure. Since $m$ is commutative up to homotopy, the induced product map on cohomology 
    $$H^*(\BBar(P))\otimes H^*(\BBar(P)) \to H^*(\BBar(P)\otimes \BBar(P)) \xrightarrow{H^*(m)} H^*(\BBar(P))$$
    is a commutative product, where the first arrow is the standard K\"unneth morphism. The fact that the product respects the weight filtration at the chain level implies that the same holds for cohomology.
    (The same construction cannot be used to get a coalgebra structure on $H^*_{\BBar}(A)$ since the K\"unneth morphism goes in the wrong direction and is not invertible in general.)
\end{proof}
\begin{corollary}
    Suppose $R$ is Noetherian. Then the contravariant functor $X\mapsto H^*_{\BBar}(X;R)$ sends pointed topological spaces to filtered commutative algebras. 
\end{corollary}

We do not have an explicit description of the algebra structure on $H^0_{\BBar}(X;R)$, but the product on its associated graded is determined by the following uniqueness statement.

\begin{lemma}\label{lem:shuffle is unique}
    Let $M$ be a flat $R$-module, and let $T^c(M)$ be a the cofree coalgebra on $M$. Then the shuffle product is the unique weight-graded product that makes $T^c(M)$ into a Hopf algebra. More generally, if $C\into T^c(M)$ is an injective homomorphism of weight-graded coalgebras then there exists at most one product structure making $C$ into a weight-graded bialgebra.

    In particular, $R\langle S_1,\ldots,S_n\rangle$ with the shuffle product is the unique cofree weight-graded Hopf algebra cogenerated by the free $R$-module with basis $S_1,\ldots,S_n$.
\end{lemma}
\begin{proof}
    First, recall that the standard shuffle product indeed makes $T^c(M)$ into a graded Hopf algebra \cite[Proposition 1.3.2]{Loday-Vallette}. We claim that any other graded Hopf structure must coincide with this one.
    
    Let $m:T^c(M)\otimes T^c(M)\to T^c(M)$ be an arbitrary graded product that makes $T^c(M)$ into a graded Hopf algebra. This means that $m$ is a coalgebra homomorphism, where $T^c(M)\otimes T^c(M)$ is given the diagonal coalgebra structure $\Delta(a\otimes b) = \sum (a_{(1)}\otimes b_{(1)})\otimes (a_{(2)}\otimes b_{(2)})$ in Sweedler notation.
    
    But since $T^c(M)$ is cofree, a  coalgebra homomorphism into it is determined by the result of post-composition with the projection $pr:T^c(M)\to M$ onto the cogenerators.  Because the cogenerator projection is graded, the fact that there exists only one graded product follows from the observation that there exists a unique graded map
    \[
    T^c(M)\otimes T^c(M)\to M
    \]
    that satisfies the unit axiom. 
    Indeed, in degree $1$, the map
    \[
    \left(T^c(M)\otimes T^c(M)\right)_1 = (R\cdot 1\otimes M) \oplus  (M\otimes R\cdot 1) \longrightarrow M
    \]
    is uniquely determined by the requirement that $1\otimes x \mapsto x$ and $x \otimes 1 \mapsto x$, and in all other degrees the target is trivial. Since the shuffle product $\ast$ satisfies these equations, it must coincide with $m$.

    For graded sub-coalgebras $i:C\into T^c(M)$ the same argument works with little change. A graded bialgebra structure on $C$ is a homomorphism of graded coalgebras $m:C\otimes C\to C$. Composing with the inclusion, $i\circ m:C\otimes C\to T^c(M)$ is uniquely determined by its (graded) projection to $M$, which is determined by the unit axiom. If $m':C\otimes C\to C$ is another graded product compatible with the coproduct, then $i\circ m=i\circ m'$, so the fact that $i$ is assumed to be injective implies that $m=m'$.
\end{proof}

\subsection{The leading-term Milnor--Hopf algebraic invariant}\label{S:primary}
Our strategy for studying links is to extract invariants from the quasi-isomorphism types of the dg-Hopf algebras $\BBar(-)$ of their complements. The first and most important invariant is the associated graded of $H^0(\BBar(-))$, which we show is determined by a labeled oriented link without any ambiguity.

This leading-term invariant will be a functorial graded submodule $\Primary(-;R)\into T^c(H^1(-;R))$ inside the cofree Hopf algebra on the $R$-module $H^1(-;R)$, obtained from the associated graded of $H^0_{\BBar}(-;R)$. A technical complication is that, depending on the ring $R$, the submodule may or may not be a sub-Hopf algebra. We therefore divide the discussion into two parts, starting with the more structured variant when $R$ is a PID, and later considering more general rings.

\begin{proposition}\label{prop:H0Bar is Hopf}
    Suppose $R$ is a PID. Then the functor $X\mapsto H^0_{\BBar}(X;R)$ from the category of pointed path-connected topological spaces takes values in commutative conilpotent Hopf algebras equipped with an identification of their primitive elements, that is, the elements of weight 1: in this case\footnote{Note that elements of $\BBar(A)$ of weight 1 (i.e.\ tensor products with only one factor) and total degree $0$ must be elements in $A^1$.} by the natural inclusion $H^1(X;R)\into H^0_{\BBar}(X;R)$.

    The associated graded of $H^0_{\BBar}(X;R)$, denoted $\Primary(X;R)$, is similarly functorial in pointed continuous maps and the bar spectral sequence
    gives rise to a natural 
    embedding of graded Hopf algebras $\Primary(X;R)\into T^c(H^1(X;R))$ whose target is the cofree tensor coalgebra equipped with the deconcatenation coproduct and the shuffle product.
    
    Maps $f:X\to Y$ induce pullbacks $f^*:\Primary(Y;R)\to \Primary(X;R)$ compatible with their respective inclusion into $T^c(H^1(Y;R))\to T^c(H^1(X;R))$.
\end{proposition}

We prove this proposition after applying it to define
our first invariant.

\begin{definition}[Leading-term Milnor--Hopf algebra with PID coefficients]
    \label{D:leading term PID}
    For any PID $R$ and a pointed path-connected space $X$, let $\Primary(X;R)$ denote the associated graded Hopf algebra of $H^0_{\BBar}(X;R)$ with respect to the weight filtration. 
    
    The {\bf leading-term Milnor--Hopf algebraic invariant}, or {\bf leading-term invariant} for short, is the inclusion of associated graded Hopf algebras
    \[
    \Primary(X;R)\into T^c(H^1(X;R)),
    \]
    where the target is the (unique) cofree graded Hopf algebra cogenerated by $H^1(X;R)$ equipped with the shuffle product.
    We will often abuse notation and write $\Primary(-;R)$ for the image of the associated graded in $T^c(H^1(-;R))$.
\end{definition}

Two corollaries follow immediately from \cref{prop:H0Bar is Hopf}.

\begin{corollary}\label{cor:equalH^1 means equal C}
    In the setting of Proposition \ref{prop:H0Bar is Hopf}, if $X\rightrightarrows Y$ are two pointed maps that induce equal pullbacks on $H^1(-;R)$, then they induce equal pullbacks $\Primary(Y;R)\rightrightarrows \Primary(X;R)$. 
    \qed
\end{corollary}

\begin{corollary}[Changing basepoints]\label{cor:change basepoints}
    Let $X$ be a path-connected topological space with two points $x_0,x_1\in X$. Suppose there exists a map of pointed spaces $f:(X,x_0)\to (X,x_1)$ that induces the identity map on $H^1(X;R)$, e.g. when $f$ is freely homotopic to the identity. Then $f$ induces the identity map on the submodules $\Primary(X;R)\into T^c(H^1(X;R))$.
    \qed
\end{corollary}

\begin{remark}
    Together, Corollaries \ref{cor:equalH^1 means equal C} and \ref{cor:change basepoints} show that for connected spaces, 
    induced maps $\Primary(Y;R)\to \Primary(X;R)$ do not depend on the chosen basepoints. That is, if $f:X\to Y$ is an unpointed map of connected spaces, then for any two basepoints $x_0,x_1\in X$, the pointed maps $f_i:(X,x_i)\to (Y,f(x_i))$ induces equal maps on $\Primary(-;R)$.
\end{remark}

\begin{remark}
    Unlike the associated graded algebra $\Primary(X;R)$, there is no natural inclusion $H^0_{\BBar}(X;R)\into T^c(H^1(X;R))$. For example, one cannot eliminate the dependence of the choice of basepoint, and indeed, pointed maps that are freely homotopic often induce distinct morphisms.
\end{remark}

To prove \cref{prop:H0Bar is Hopf}, we use the weight filtration, its associated spectral sequence, and its compatibility with the coproduct.  The following lemma
is as expected, but we include its proof for completeness.

\begin{lemma}\label{lem:spectral sequence}
    Let $R$ be an arbitrary commutative ring and suppose $A$ is a dg-algebra that is a flat $R$-module such that $H^0(A)=R$ and $H^i(A)=0$ for $i<0$. Suppose furthermore either
    \begin{itemize}
        \item $H^*(A)$ is flat over $R$, or
        \item $R$ is a PID and $H^1(A)$ is torsion-free\footnote{Recall that if $R$ is a PID, then an $R$-module is flat if and only if it is torsion-free \cite[Proposition 4.20]{LAM}.}.
    \end{itemize}
    Then the map $\Delta:\BBar(A)\to \BBar(A)\otimes \BBar(A)$ strictly respects the weight filtration, assigning to an element of $\BBar(A)\otimes \BBar(A)$ the weight given by the sum of the weights in each tensor factor, and it induces a map between the respective weight spectral sequences.
    Both spectral sequences have their $E_1$ page concentrated in nonnegative total degrees and, furthermore, in total degree zero there is a commutative diagram
    \[
    \xymatrix{
    E_1(\BBar(A))^0 \ar[r]^-{E_1(\Delta)} \ar@{=}[d] & E_1(\BBar(A)\otimes \BBar(A))^0 \ar@{=}[d] \\
    T^c(H^1(A)) \ar[r]^-\Delta & T^c(H^1(A))\otimes T^c(H^1(A)),
    }
    \]
    where $T^c(-)$ is the cofree tensor coalgebra.
\end{lemma}
\begin{proof}
    The strict preservation of the weight filtration by $\Delta$ is clear by direct inspection. Since $\Delta$ furthermore satisfies the Leibniz rule, it is a chain map and thus induces a map of the respective spectral sequences.

    The $E_1$ page of the weight spectral sequence for $\BBar(A)$ takes the form
    \[
    E_1^{-p,q} = H^q(\bar A^{\otimes p}).
    \]
    Let us therefore focus on the cohomology of tensor powers. Since $A$ is a complex of flat $R$-modules, its summand $\bar A$ and all its tensor powers are flat. One can therefore compute $H^*(\bar A\otimes \bar A^{\otimes p-1})$ using the K\"unneth spectral sequence \cite[Theorem 2.20]{McC}, for which
    \begin{equation}\label{eq:Kunneth for E1 page}
    E_2^{s,t} = \bigoplus_{i+j=t}\operatorname{Tor}^R_{s}( H^i(\bar A),H^j(\bar A^{\otimes p-1})) \implies H^{t-s}(\bar A^{\otimes p}).
    \end{equation}
    
    If $H^*(A)$ is flat over $R$, then so is its summand $H^*(\bar A)$ and we conclude that the K\"unneth summands with $s>0$ vanish, inductively obtaining isomorphisms
    \[
    H^*(\bar A^{\otimes p}) \cong H^*(\bar A)\otimes H^*(\bar A^{\otimes p-1}) \cong \cdots \cong H^*(\bar A)^{\otimes p}.
    \]
    So the $E_1$ page of the weight spectral sequence of $\BBar(A)$ has columns $H^*(\bar A)^{\otimes p}$, and we see that its $d_1$-differential is given by the cohomological product of $A$, meaning that $E_1 = \BBar(H^*(\bar A))$. But since $H^0(A) = R$ and $H^i(A)=0$ for $i<0$, the augmentation ideal has $H^i(\bar A)=0$ for $i\leq 0$, and the entire $E_1$-page vanishes in negative total degree. In total degree zero the only contribution to $\BBar(H^*(A))$ is from tensors of $H^1(A)$, so we have an isomorphism $E_1(\BBar(A))^0\cong T^c(H^1(A))$ as claimed.

    The second case we consider is when $R$ is a PID and $H^1(A)$ is torsion-free. Returning to the K\"unneth spectral sequence for $H^*(\bar A^{\otimes p})$, we see that the terms in formula \eqref{eq:Kunneth for E1 page} with $s>1$ vanish, since $R$ is a PID.  By assumption, $H^{1}(\bar A)$ is flat and $H^i(\bar A)=0$ for $i \leq 0$, so $\operatorname{Tor}^R_s(-,-)$ terms with $s=1$ vanish on such arguments as well. Proceeding by induction on $p$, we deduce that $H^{i}(\bar A^{\otimes p}) = 0$ for $i<p$ and $H^p(\bar A^{\otimes p}) \cong H^1(A)^{\otimes p}$ as before. Therefore, we get the same description of the $E_1$ page of the weight spectral sequence in total degree zero.

    The same argument applies to $\BBar(A)\otimes \BBar(A)$ to give a description of the $E_1$-page as vanishing in negative total degrees and as tensors of $H^1(A)$ in total degree zero.

    Lastly, for the coproduct structure on the $E_1$-page, consider the deconcatenation component $\Delta_i:\bar A^{\otimes p}\to \bar A^{\otimes i}\otimes \bar A^{\otimes p-i}$. Tracing through the action on cocycles, there is a commutative diagram
        \[
        \xymatrix{
        H^*(\bar A)^{\otimes p} \ar[r]^-{\Delta_i} \ar@{->>}[d] & H^*(\bar A)^{\otimes i} \otimes H^*(\bar A)^{\otimes p-i} \ar@{->>}[d] \\
        H^*(\bar A^{\otimes p}) \ar[r]^-{H^*(\Delta_i)} & H^*(\bar A^{\otimes i}\otimes \bar A^{\otimes p-i})
        }
        \]
    whose vertical arrows are isomorphisms in degree $p$ and so uniquely characterizes the induced map in total degree zero on the $E_1$-page as deconcatenation.
\end{proof}

With the technical input of \cref{lem:spectral sequence}, we can now  prove the first main result of this subsection.

\begin{proof}[Proof of Proposition \ref{prop:H0Bar is Hopf}]

    In outline, we first obtain an inclusion of $R$-modules $\Primary(X;R) \hookrightarrow T^c(H^1(X;R))$ and show that it is a map of coalgebras; then we check that $H^0_{\BBar}(X;R)$ is a Hopf algebra with the claimed cogenerators; and finally we show that the inclusion is a map of algebras.
    Let $A:=C^*(X;R)$ denote a differential graded algebra of cochains on a pointed path-connected space $X$ with coefficients in a PID. 
    
    We first claim that $H^0(\BBar(A))$ is $R$-torsion-free. Indeed, since cochains are $R$-torsion-free and $H^1(X;R) \cong \Hom(H_1(X),R)$ is similarly torsion-free, Lemma \ref{lem:spectral sequence} shows that the bar spectral sequence computing $H^0(\BBar(A))$ has $E_1$ vanishing in negative total degree and torsion-free in total degree zero. Since no nontrivial differentials may hit elements of total degree zero, there is an inclusion
    \[
    E_\infty^{-p,p} \into E_1^{-p,p} \cong H^1(A)^{\otimes p}
    \]
    whose image is characterized by the vanishing of all higher differentials. Explicitly, the $E_\infty$-page is the associated graded of $H^0_{\BBar}(A)$, denoted here by $\Primary(X;R)$, and the inclusion in total degree zero $E_{\infty}\into E_1$ gives a natural embedding $\Primary(X;R)\into T^c(H^1(X;R))$, 
    which establishes the claimed embedding.  
    
    In weight one, $E_1^{-1,1}\cong H^1(A)$, and this group supports no nontrivial differentials, so it survives to $E_\infty$.  Thus the composite $\Primary(X;R)\into T^c(H^1(A))\xrightarrow{pr} H^1(X;R)$ restricts to an isomorphism on weight-one elements. 
    Additionally, $\Primary(-;R)$ is functorial in continuous maps as a composition of several functors, and the embedding $\Primary(-;R) \into T^c(H^1(-;R))$ is natural because it comes from a map of spectral sequences.

    To address  coalgebra structure, we consider the K\"unneth morphism
    \begin{equation}
    \label{eq:Kunneth morphism}
    H^0(\BBar(A))\otimes H^0(\BBar(A)) \to H^0(\BBar(A)\otimes \BBar(A)).
    \end{equation}
    We claim that it is an isomorphism and thus gives a unique coproduct on $H^0_{\BBar}(A)$ compatible with the induced map $\Delta:\BBar(A)\to \BBar(A)\otimes \BBar(A)$.

    To verify this claim, first note that as $\Primary(X;R)$ is a submodule of the $R$-torsion-free module $T^c(H^1(A))$, it is itself torsion-free. Thus the same holds for $H^0(\BBar(A))$. Now since $\BBar(A)$ is $R$-torsion-free, the standard K\"unneth formula for PIDs applies to compute $H^0(\BBar(A)\otimes \BBar(A))$ as an extension of 
    $\bigoplus_{i+j=1}\operatorname{Tor}^R_1 (H^i(\BBar(A)),H^j(\BBar(A)))$
    by 
    $\bigoplus_{i+j=0}H^i(\BBar(A))\otimes H^j(\BBar(A))$.
    There is no cohomology in strictly negative degrees so the only $\operatorname{Tor}$ terms that may contribute to the zeroth cohomology are ones in which at least one of the arguments is $H^0(\BBar(A))$. But the latter is a torsion-free $R$-module, so all $\operatorname{Tor_1}$-terms vanish and the K\"unneth morphism \eqref{eq:Kunneth morphism} is an isomorphism. 
    
    Inverting the K\"unneth morphism gives the natural 
    coalgebra structure on $H^0_{\BBar}(A)$ by the composite
    \[
    H^0(\BBar(A)) \xr{H^*(\Delta)} H^0(\BBar(A)\otimes \BBar(A)) \xrightarrow{\cong} H^0(\BBar(A))\otimes H^0(\BBar(A)).
    \]
    The coalgebra axioms can be checked using those for $\BBar(A)$ and the naturality of the K\"unneth map.
    A  posteriori, it follows that the inclusion $\Primary(A)\into T^c(H^1(A))$ is a homomorphism of coalgebras, as both structures are induced from the coproduct on $\BBar(A)$.

    The coproduct on $H^0(\BBar(A))$ is compatible with the product from \cref{lem:product on bar} by the compatibility on the quasi-isomorphic model given in that lemma and the naturality of the coproduct. 
    Hence $H^0_{\BBar}(A)$ is a commutative bialgebra. Now $H^0_{\BBar}(A)$ is augmented and 
    connected relative to the weight filtration, so it admits a unique antipode making it into a Hopf algebra.

    Before taking cohomology, the inclusion of weight one $\bar A\into \BBar(A)$ gives a natural identification of the primitives of $\BBar(A)$. The bar spectral sequence shows that $H^1(A)\to H^0_{\BBar}(A)$ is injective, and it continues to surject onto the primitive elements, thus giving an isomorphism
    \[
    H^1(A)\xrightarrow{\sim} \mathrm{Prim}(H^0_{\BBar}(A))
    \]
    that is natural in homomorphisms $A\to A'$.


    Lastly, since the product structure on $\BBar(A)$ propagates through the bar spectral sequence, the inclusion $\Primary(A)=E_\infty\into E_1=T^c(H^1(A))$ is a homomorphism of graded algebras. Lemma \ref{lem:shuffle is unique} states that the algebra structure on $T^c(H^1(A))$ must be the shuffle product. 
\end{proof}

To later relate $\Primary(-;R)$ of a link complement with Milnor's $\bar \mu$-invariants, we must also consider the bar construction with coefficients in the rings $R=\Z_m$, which are not domains in general. One might hope that the coalgebra structure $\Delta:\BBar(A)\to \BBar(A)\otimes \BBar(A)$ induces a coalgebra structure on its cohomology as well. This is not true in general, as the following example shows.
\begin{example}[Non-coalgebra]\label{ex:non-coalgebra}
    Consider the augmented dg-coalgebra defined over $R=\Z_4$
    \[
    (C_0\to C_1\to C_2) = (R\cdot 1\xrightarrow 0 R\cdot x\xrightarrow{2} R\cdot y)
    \]
    with augmentation ideal $\bar C$ and reduced coproduct $\barDelta(x)=0$ and $\barDelta(y)=2x\otimes x$. Then $H^1(\bar C) = \langle2x\rangle$ and $H^2(\bar C\otimes \bar C) = \langle 2x\otimes x\rangle$, which means that the K\"unneth map $H^1(C)\otimes_R H^1(C)\to H^2(C\otimes C)$ is the zero map, though the domain is nontrivial.  But the coproduct $\langle y\rangle = H^2(\bar C)\xrightarrow{\barDelta_*} H^2(\bar C\otimes \bar C) = \langle2x\otimes x\rangle$ is nonzero. Thus the coproduct on $C$ does not induce a coproduct structure on its cohomology. 

\end{example}
While $H^0_{\BBar}(X;R)$ is not guaranteed to form a coalgebra, in our context it is nonetheless possible to embed the algebra $\Primary(X;R)$ in a cofree Hopf algebra. As such it gives essentially the same information as when $R$ is a PID.

\begin{proposition}\label{prop:primary over nonPID}
    Let $R$ be a commutative Noetherian ring. Consider the functor $X\mapsto \Primary(X;R)$ on the full subcategory of pointed path-connected topological spaces for which $H^*(X;R)$ is flat over $R$. It admits a natural graded embedding of commutative algebras $\Primary(X;R)\into T^c(H^1(X;R))$, whose target is equipped with the shuffle product.
\end{proposition}
The analogs of Corollaries \ref{cor:equalH^1 means equal C} and \ref{cor:change basepoints} hold in this context as well.
\begin{corollary}
    Let $R$ be a commutative Noetherian ring and suppose $X$ and $Y$ are pointed path-connected spaces for which $H^*(-;R)$ is a flat $R$-module. If $X\rightrightarrows Y$ are two pointed maps that induce equal pullbacks on $H^1(-;R)$ then they induce equal pullbacks $\Primary(Y;R)\rightrightarrows \Primary(X;R)$.
    \qed
\end{corollary}

\begin{proof}[Proof of Proposition \ref{prop:primary over nonPID}]
    Let $X$ be a path-connected topological space with $H^*(X;R)$ flat over $R$ and set $A=C^*(X;R)$. The module $\Primary(X;R)$ is the associated graded of $H^0_{\BBar}(A)$ and coincides with the $E_\infty$-page of the weight spectral sequence for $\BBar(A)$.

    In every cohomological degree $q$ the cochain group $C^q(X;R)$ is isomorphic to some (possibly infinite) product of the ground ring $R$. By \cite[Theorem 4.47]{LAM}, such products of Noetherian rings are flat $R$-modules\footnote{For geometric cochains, $C^*_\Gamma(M)$ is always flat by \cref{L: goeometric flat}.}, so Lemma \ref{lem:spectral sequence} applies and gives an embedding in total degree zero
    \[
    \Primary(A) = (E_{\infty})^0 \into (E_1)^0 \cong T^c(H^1(A)).
    \]
    
    In \cref{lem:product on bar} we exhibited a quasi-isomorphism of filtered coalgebras $\BBar(P)\to \BBar(A)$ where the domain is equipped with a homotopy commutative algebra structure. A filtered quasi-isomorphism, in particular, induces an isomorphism of the associated spectral sequences starting with their first page. Through this isomorphism, there is a compatible algebra structure on all pages of the weight spectral sequence for $\BBar(A)$, starting with $E_1$ and converging to the algebra structure on $H^0_{\BBar}(A)$. The inclusion $\Primary(A)\into T^c(H^1(A))$ is therefore a homomorphism of graded algebras.


    It remains to show that the product structure on $E_1(\BBar(A))\cong T^c(H^1(A))$ is the standard shuffle product. By \cref{lem:spectral sequence}, our flatness assumptions endow this page with the cofree colagebra structure even when later pages of the spectral sequence are not coalgebras. The characterization now follows from uniqueness of a graded bialgebra structure on $T^c(H^1(A))$ proved in Lemma \ref{lem:shuffle is unique}. 
    
\end{proof}
\begin{definition}[Leading-term Milnor--Hopf algebra with non-PID coefficients]
    Let $R$ be a commutative Noetherian ring and $X$ a pointed space for which $H^*(X;R)$ is flat over $R$. Denote the associated graded of the filtered algebra $H^0_{\BBar}(X;R)$ by $\Primary(X;R)$. The {\bf leading-term Milnor--Hopf algebraic invariant} of $X$ is the inclusion of graded algebras
    \[
    \Primary(X;R)\into T^c(H^1(X;R)),
    \]
    where the target is the cofree graded Hopf algebra on $H^1(X;R)$ equipped with the shuffle product, but the image of $\Primary(X;R)$ is not necessarily closed under the coproduct.

    As with \cref{D:leading term PID}, we will often abuse notation and write $\Primary(-;R)$ for the image of the associated graded in $T^c(H^1(-;R))$.
\end{definition}

Specializing the construction of $\Primary(X;R)$ to the case when $X$ is a link complement, we obtain a  concrete description by employing standard Alexander duality, which
we recall in the following. 

\begin{lemma}[Alexander Duality]\label{lem:primary embedding for links}
Let $R$ be a commutative Noetherian ring, $M$ an $R$-homology $d$-sphere, and $N$ an $R$-oriented $(d-2)$-manifold consisting of $n$ connected components. For example, take $N=nS^1$ and $M=S^3$. For a link $L:N\into M$ there is a canonical isomorphism
\[
H^1(M\setminus L;R) \cong R^n,
\]
with basis $S_1,\ldots,S_n\in H^1(M\setminus L;R)$, where $S_i$ is Alexander dual to the $i$-th link component. Equivalently, $S_i$ is the geometric cohomology class represented by a Seifert hypersurface for the $i$-th link component (that is,  any oriented immersed hypersurface\footnote{In fact, more generally than classical Seifert surfaces, $S_i$ need only be immersed into the ambient manifold $M$, transverse to $L$ away from its boundary, and then the restriction $S_i\setminus L \into M\setminus L$ represents a geometric cohomology class. 
} 
with boundary equal to the $i$-th component as an $R$-oriented manifold of codimension $2$). 
Moreover, $H^*(M\setminus L;R)$ is free over $R$ and, in particular, flat.

If $f:M\to M'$ is a homeomorphism, then $f$ takes the components of the link $L$ to components of the link $L':=f\circ L$. Identifying each set of respective dual classes with the set $\{S_1,\dots, S_n\}$, we
obtain a commutative triangle
\[
\xymatrix{
H^1(M'\setminus L';R) \ar[rr]^-{f^*} \ar[rd] & &  H^1(M\setminus L;R) \ar[ld]\\
& \operatorname{Span_R}(S_1,\ldots,S_n). &\qedhere
}
\]
\qed
\end{lemma}

The following two corollaries settle the proofs of \cref{T: concordance invariants} and \cref{intro-thm:primary invariant}, showing that the bar construction takes values in marked Hopf algebras and that the Milnor--Hopf invariant is rigid.

\begin{corollary}
    Let $R$ be a PID and $L:N\into M$ an $n$-component link in an $R$-homology $d$-sphere. The Hopf algebra $H^0_{\BBar}(M\setminus L;R)$ is marked in the sense of \cref{def:marked} by
    \[
    R^n \cong H^1(M\setminus L;R) \cong \mathrm{Prim}(H^0_{\BBar}(M\setminus L;R))
    \]
    with the basis $S_1,\ldots,S_n$ from \cref{lem:primary embedding for links}. A homeomorphism $f:M\to M'$ induces a homomorphism of marked Hopf algebras $H^0_{\BBar}(M'\setminus f\circ L;R)\to H^0_{\BBar}(M\setminus L;R)$.
    \qed
\end{corollary}

\begin{corollary}
    For any commutative Noetherian ring $R$ and $L:N\into M$ an $n$-component link in an $R$-homology $d$-sphere, there is an identification
    \[
    T^c(H^1(M\setminus L;R)) \cong R\langle S_1,\ldots,S_n\rangle
    \]
    functorial in homeomorphisms 
    of the ambient space. Thus the leading-term invariant comes with a functorial embedding
    \[
    \Primary(M\setminus L;R)\into R\langle S_1,\ldots,S_n\rangle
    \]
    and is independent of a choice of exterior basepoint and of Seifert hypersurfaces.
    
    Consequently, if $L_1:N\into M_1$ and $L_2:N\into M_2$ are links related by a homeomorphism of the ambient manifolds, then $\Primary(M_1\setminus L_1;R) = \Primary(M_2\setminus L_2;R)$ as submodules of $R\langle S_1,\ldots,S_n\rangle$.
    \qed
\end{corollary}
In a sequel to this work we extend this functoriality from ambient homeomorphisms to arbitrary topological concordances.
\begin{definition}[Leading-term link invariant]
    For an $n$-component link $L:N\into M$ in an $R$-homology $d$-sphere, and when the coefficient ring $R$ is clear from context, the {\bf Milnor--Hopf leading-term invariant} of the link is the image of the functorial submodule $\Primary(M\setminus L;R)$ in $R\langle S_1,\ldots,S_n\rangle$, denoted by $\Primary(L)$.
\end{definition}

We can now prove \cref{P: associated graded augmentation}, which relates the leading-term invariants of links to their associated link groups.
Recall that $I(L)$ denotes the augmentation ideal of the group ring $R[\pi_1(M\setminus L)]$, omitting the coefficients from the notation.

\begin{proof}[Proof of \cref{P: associated graded augmentation}]
    Let $R$ be a field, and let $L_0$ and $L_1$ be two links in an $R$-homology sphere $M$ such that $\Primary(L_0)=\Primary(L_1)$.  By work of the second author \cite[Theorem 1.3]{Gad23}, there are natural isomorphisms of filtered coalgebras
    \begin{equation}
    \label{eq:H0Bar vs I-adic completion}
        \varinjlim_{p\to \infty} \Hom(R[\pi(L_i)]/I^{p+1},R) \cong H^0_{\BBar}(L_i;R).
    \end{equation}
    In particular, $\Hom(R[\pi(L_i)]/I^{p},R) \cong H^0_{\BBar}(L_i;R)_{\leq p}$, which implies an isomorphism of associated graded coalgebras 
    \[
    \bigoplus_{p=0}^{\infty} \Hom(I(L_i)^p/I(L_i)^{p+1},R) \cong \Primary(L_i).
    \]
    Since both objects in the last isomorphism are in fact graded Hopf algebras and the leading-term invariant is a sub-Hopf algebra of $R\langle S_1,\ldots, S_n\rangle$, the uniqueness statement for graded Hopf structures forces the product structures to coincide.
    
    An equality of leading-term invariants thus determines an isomorphism of graded Hopf algebras
    \[
    \bigoplus_{p=0}^{\infty} \Hom(I(L_0)^p/I(L_0)^{p+1},R) \cong     \bigoplus_{p=0}^{\infty} \Hom(I(L_1)^p/I(L_1)^{p+1},R)
    \]
    that restricts to an identification 
    \[
    \Hom (I(L_0)/I(L_0)^2,R) \cong \mathrm{Span}_R(S_1,\ldots,S_n) \cong \Hom(I(L_1)/I(L_1)^2,R)
    \]
    in weight one.

    Computing the graded $R$-linear duals and using the finite-dimensionality of each quotient $I^p/I^{p+1}$, we get an isomorphism
    \[
    \bigoplus_{p=0}^{\infty}  I(L_0)^p/I(L_0)^{p+1} \cong     \bigoplus_{p=0}^{\infty} I(L_1)^p/I(L_1)^{p+1}.
    \]
    Observing that the meridian classes $(\mathpzc{m}_k(L)-1) \in I(L)/I(L)^2$ with $k=1,\dots,n$ form a dual basis to $(S_1,\ldots,S_n)$, we see that this isomorphism must preserve such meridians.

    Conversely, an isomorphism of the Hopf algebras     $\bigoplus_{p=0}^{\infty}  I(L_i)^p/I(L_i)^{p+1}$ respecting meridians dualizes to an isomorphism $\Primary(L_0)\cong \Primary(L_1)$ that commutes with evaluation on meridian classes. But this evaluation defines the cogenerator projections $\Primary(L_i)\to \operatorname{Span_R(S_1,\ldots,S_n)}$, which in turn defines the inclusion $\Primary(L_i)\into R\langle S_1,\ldots,S_n\rangle$. The isomorphism of leading-term invariants is thus compatible with the respective injections, identifying $\Primary(L_i)$ as submodules of $R\langle S_1,\ldots,S_n\rangle$.
\end{proof}

\section{Geometric calculations of the leading-term invariant}\label{S:examples}

Our preferred computational method for verifying membership of non-commuting polynomials $P(S_1,\ldots,S_n)\in R\langle S_1,\ldots,S_n\rangle$ in the leading-term Milnor--Hopf algebraic link invariant $\Primary(L)$ is by constructing systems of hypersurfaces that represent cocycles in $Bar^0(M\setminus L;R)$. For concreteness, we develop this theory only in three dimensions, but the same construction holds verbatim in arbitrary dimension. This approach makes explicit use of the theory of geometric cochains $C^*_\Gamma(-)$ developed by authors GF and DS together with Medina-Mardones \cite{FMS-foundations}, recalled briefly in \cref{A:geometric}.

Each cochain algebra model of a space -- in particular singular, geometric, or de Rham - could potentially give a distinct link invariant through its bar cohomology.
Contrasting bar invariants of links arising from the geomteric and singular models, the former is immediately computable using the geometry of the link complement, while the latter has been proven to capture the dimension filtration of
the fundamental group.

Models related through an $A_\infty$-quasi-isomorphism yield isomorphic  bar invariants. Such a quasi-isomorphism from the de Rham to the singular complex was constructed by Gugenheim \cite{Gu76} using Chen integrals. The following comparison question is open.

\begin{conjecture}
    There exists a natural $A_{\infty}$-quasi-isomorphisms 
    \[
    C^*_\Gamma(-) \longrightarrow C^*_{\operatorname{Sing}}(-;\Z).
    \]
    Equivalently, the respective bar constructions are naturally quasi-isomorphic as dg coalgebras.
\end{conjecture}

The $\Z/2$ case of this is resolved in the forthcoming dissertation work of A.~Pizzi \cite{Pizzi}.  The proof
uses a zig-zag of quasi-isomorphisms to a bivariant
theory, which seems to pass immediately to all field
coefficients. 

\subsection{Making calculations with surface systems}\label{surface systems}

For calculations, \cref{prop:geometric criterion} gives a geometric criterion for membership in $\Primary(L)$ using  surface systems, a distilling of the combinatorics
and geometry of cocycles in the geometric cochain
bar complex which we now develop. The  definition in generality is necessarily complicated, but  regularly referencing the simple cases in Examples~\ref{ex:pairwise} and \ref{ex:triple} should facilitate understanding.

In this section, our coefficient ring $R$ is always $\Z$ or $\Z_m$.
We begin by recording the orientation conventions consistent with geometric cochains as established in \cite[Section 8.2.1]{FMS-foundations}. A Seifert surface $S$ with oriented boundary link component $L$ is oriented so that, at $L$, an outward pointing normal tangent to $S$ followed by the orientation of $L$ agrees with the orientation of the ambient $S$. This is consistent with the standard ``counterclockwise’’ rule for planar regions. 
If link component $L_1$ pierces Seifert surface $S_2$ so that the orientation of $L_1$ followed by the orientation of $S_2$ agrees with the orientation of $M$, then the sign of this intersection point of $L_1\cap S_2$ is +1, and the corresponding intersection curve of $S_1\cap S_2$ is oriented away from $L_1$. Consequently, the total linking number of $L_1$ with $L_2$ is equal to the number of intersection curves of $S_1\cap S_2$ oriented from $L_1$ to $L_2$ minus the number of oriented intersection curves from $L_2$ to $L_1$.  

\begin{definition}
A set of maps $f_\alpha : \Sigma_\alpha \to M$, $\alpha \in J$ is {\bf fully transverse} if (inductively) for any disjoint subsets
$J_{i_1}, \cdots J_{i_k} \subset J$, the maps from 
their fiber products to $M$ are transverse.  
If $\Sigma_\alpha$ have boundary, we require transversality
for all  fiber products over all possibilities where
each map could
be $f_\alpha$ or could be its restriction to
the boundary. 
\end{definition}

\begin{definition}[Surface systems]
    Let $L:nS^1\into M$ be a link. An {\bf incomplete surface system} of weight $p$ for $L$ is a collection of fully transverse tuples $\bigsqcup_{\ell=2}^p \{(\Sigma_1^\alpha,\ldots ,\Sigma_\ell^\alpha)\}_{\alpha\in I_\ell}$ of immersed surfaces, where $I_2, \dots, I_p$ are some finite index sets, satisfying the following conditions:
    \begin{itemize}
        \item Every $\Sigma_{j}^\alpha$ is a connected, properly immersed, oriented  surface $\Sigma_{j}^\alpha\looparrowright (M\setminus L)$, possibly with corners.
        \item For all $\ell$ with $2\leq \ell \leq p$ and all $\alpha\in I_\ell$, every subset of the set of surfaces $\{\Sigma_j^\alpha\mid 1\leq j\leq \ell\}$ meets transversely in $M\setminus L$, with disjoint boundaries and loci of double-points.
        \item  Given $1\leq j\leq \ell< p$ and an $(\ell-1)$-tuple $\bar\Sigma =(\Sigma_1,\ldots,\Sigma_{\ell-1})$ of immersed surfaces, let 
        $$
        I_{j,\ell}(\bar\Sigma) := \{ \alpha\in I_\ell \mid (\Sigma_1^\alpha,\ldots,\widehat{\Sigma_j^\alpha},\ldots,\Sigma_\ell^\alpha) =\bar\Sigma \}\subseteq I_\ell$$ and
        $$
        I_{j,\ell+1}(\bar\Sigma) := \{ \alpha\in I_{\ell+1} \mid (\Sigma_1^\alpha,\ldots,\widehat{\Sigma_j^\alpha},\widehat{\Sigma_{j+1}^\alpha},\ldots,\Sigma_\ell^\alpha) =\bar\Sigma \}\subseteq I_{\ell+1}
        $$
        be the subsets of $\ell$-tuples ($(\ell+1)$-tuples) that agree with $\bar\Sigma$ outside the $j$-th ($j$-th and $j+1$-st) positions. Then for every $j$ and $\ell $ such that $1\leq j\leq \ell<p$ and for every tuple of surfaces $\bar\Sigma$, the following equality of geometric $1$-cocycles with $R$-coefficients holds:\footnote{The technical meaning of this equality is as one of geometric cochains $C^*_\Gamma(M\setminus L;R)$, which is somewhat involved. However, since these are $1$-cocycles, one can understand equality geometrically in analogy with an equality of singular chains. Sums correspond to unions of oriented arcs, negatives correspond to inverting orientation, and the zero element corresponds to the empty $1$-chain (which can be achieved through cancellations).}
        \begin{equation}\label{eq:vert diff = hors diff}
            \sum_{\alpha\in I_{j,\ell}(\bar\Sigma)} \partial \Sigma_j^\alpha  = \sum_{\beta\in I_{j,\ell+1}(\bar \Sigma)} (\Sigma_j^\beta \cap \Sigma_{j+1}^\beta).
        \end{equation}
        \item For all $p$-tuples $\alpha\in I_p$ and all $j$ such that $1\leq j\leq p$, the surface $\Sigma_j^\alpha$ is a Seifert surface for one of the link components: its boundary in $M$ is a compatibly oriented link component $L\circ i:S^1\into nS^1\to M$. Let $\operatorname{cl}(\Sigma_j^\alpha) := S_i\in H^1(M\setminus L;R)$ 
        be the cohomology class that it represents
        via intersection with (transverse) $1$-cycles.
    \end{itemize}
    We say that the {\bf leading term} of such an incomplete surface system is the cohomology class
        \[
        \sum_{\alpha\in I_p} [\operatorname{cl}(\Sigma_1^\alpha)|\ldots |\operatorname{cl}(\Sigma_p^\alpha)]\in H^1(M\setminus L;R)^{\otimes p},
        \]
    and the {\bf boundary} of the incomplete surface system is $\sum_{\alpha\in I_2} \Sigma_1^\alpha\cap \Sigma_2^\alpha$, which always consists of curves in $M\setminus L$ whose sum has no boundary, and so it represents a geometric 2-cocycle; see \cref{A:geometric}.
\end{definition}

While our leading-term expression does not explicitly include coefficients, calculations requiring coefficients can be achieved by indexing the same collection of surfaces multiple times and/or including the same surface with different orientations.

\begin{definition}
        A {\bf complete surface system} is the union of an incomplete surface system together with a collection of immersed surfaces $\{ (\Sigma_1^\alpha) \}_{\alpha\in I_1}$ indexed by some set $I_1$ such that
    \[
    \sum_{\alpha\in I_1} \partial \Sigma_1^\alpha  =\sum_{\beta\in I_2} (\Sigma_1^\beta \cap \Sigma_{2}^\beta),
    \]
    witnesses the vanishing of the cohomology class $\left[\sum (\Sigma_1^\beta \cap \Sigma_{2}^\beta)\right]\in H^2(M\setminus L;R)$.
\end{definition}

\begin{proposition}[Geometric criterion for the leading-term invariant]\label{prop:geometric criterion}
    Fix $R=\Z$ or $\Z_m$ and $M$ an $R$-homology 3-sphere. For a link $L:nS^1\into M$, a homogeneous bar polynomial
    \[
    \sum_{i_1,\ldots,i_p} r_{i_1\cdots i_p} \cdot[S_{i_1}|\cdots |S_{i_p}] \in R\langle S_1,\ldots S_n\rangle,
    \]
    belongs to the leading-term invariant $\Primary(L)$ whenever there exists a complete surface system with this leading term.

    Moreover, suppose that one can find an { incomplete surface system} with this leading term and boundary $\sum (\Sigma_1^\alpha\cap \Sigma_2^\alpha)$.
    Then the leading term $\sum_{i_1,\ldots,i_p} r_{i_1\cdots i_p} \cdot[S_{i_1}|\cdots |S_{i_p}]$ belongs to $\Primary(L)$ if and only if there exists some element $T\in \BBar(C^*_\Gamma(M\setminus L;R))$ of weight $<p$ with total boundary of weight one $d_{\BBar}T\in C^2(M\setminus L;R)$ representing the cohomology class
    $[\sum_{\alpha\in I_2} \Sigma_1^\alpha\cap \Sigma_2^\alpha]\in H^2(M\setminus L;R)$.
\end{proposition}

We provide some remarks and examples before proving this proposition.

\begin{remark}[Generalized Massey products]
    The reader familiar with Massey products will notice that the criterion given in the last proposition is a generalization of the standard formula for the vanishing of a Massey product. Indeed, by transferring the associative product on $C^*(-)$ to an $A_\infty$-structure on $H^*(-)$, cocycles in the bar construction can be identified with tensors of cohomology classes on which all Massey products vanish, but where non-homogeneous linear combinations are allowed.

    We prefer a hands-on geometric approach using surfaces since we do not have a more direct way to compute higher Massey products.
\end{remark}

\begin{remark}\label{rmk:algorithm}
One can attempt to construct these using the following {\bf surface system algorithm}:
\begin{enumerate}[label = Step \arabic*),leftmargin=2.3cm]
    \item To show that $\sum \lambda_\alpha [S_{\alpha_1}|\ldots|S_{\alpha_p}]$ belongs to $\Primary(L)$, with $\lambda_\alpha \in R=\Z$ or $\Z_m$ pick a Seifert surface $\Sigma_i\looparrowright M$ for every link component, so $\Sigma_i$ represents the cohomology class $S_i\in H^1(M\setminus L;R)$. Then every tensor $[S_{\alpha_1}|\ldots|S_{\alpha_p}]$ is represented geometrically by a $p$-tuple of surfaces $(\Sigma_{\alpha_1},\ldots,\Sigma_{\alpha_p})$ and coefficients can be realized by indexing the same collection multiple times and/or considering the same surfaces with reversed orientations.
    \item When the (oriented) intersection $\Sigma_{\alpha_j}\cap \Sigma_{\alpha_{j+1}}$ bounds a properly-immersed oriented surface-with-corners in $M \setminus L$, pick such a surface $\Sigma_{j,j+1}$ that is transverse to the set of all previously chosen surfaces and record the $(p-1)$-tuples obtained by $[\ldots|\Sigma_{\alpha_j}|\Sigma_{\alpha_{j+1}}|\ldots] \rightsquigarrow [\ldots|\Sigma_{j,j+1}|\ldots]$. It might be necessary to consider linear combinations before one finds that the $j$-th intersection bounds; in particular, it will be necessary to orient properly to satisfy \eqref{eq:vert diff = hors diff}.
    \item Iterate the process for every intersection of consecutive surfaces, and on every $\ell$-tuple produced by previous iterations, producing even shorter tuples of surfaces.
    \item The process stops when all intersections have been bounded or when one arrives at a $1$-cycle that does not bound.
\end{enumerate}
If all intersections have been successfully bounded, then $\sum \lambda_\alpha [S_{\alpha_1}|\ldots|S_{\alpha_p}]\in \Primary(L)$.

Technically, we must be careful in such constructions to always choose surfaces with sufficient transversality. The machinery of geometric cochains shows that when a bar cocycle exists this can be done \cite{GBF47}. When the same Seifert surface appears multiple times in the same expression, we must use small pushoffs to obtain disjoint copies for these purposes.
\end{remark}

\begin{example}\label{ex:geometric cocycle for a 3-chain}
    To show that $([S_1|S_2]+[S_2|S_3] + [S_3|S_1])|[S_4] \in \Primary(L)$, one may proceed by first picking Seifert surfaces $\Sigma_1,\ldots,\Sigma_4$. If possible, find immersed surfaces such that:
    \[
    \partial\Sigma_{123} = (\Sigma_1\cap \Sigma_2) +(\Sigma_2\cap \Sigma_3) + (\Sigma_3\cap \Sigma_1) \quad \text{and} \quad\partial\Sigma_{i4} = \Sigma_i\cap \Sigma_4,
    \]
    then add the pairs $(\Sigma_{123},\Sigma_4)$ and $(\Sigma_1,\Sigma_{24}),(\Sigma_2,\Sigma_{34}),(\Sigma_3,\Sigma_{14})$ to the list of tuples $I_2$. This gives an incomplete surface system with boundary
    \[
    (\Sigma_{123}\cap \Sigma_4) + (\Sigma_{1}\cap \Sigma_{24})+ (\Sigma_{2}\cap \Sigma_{34})+ (\Sigma_{3}\cap \Sigma_{14}).
    \]

    The boundary cohomology class is zero in $H^2(M\setminus L;R)$ if and only if the incomplete surface system can be completed. Thus, $[S_1|S_2|S_4]+[S_2|S_3|S_4] + [S_3|S_1|S_4] \in \Primary(L)$ whenever the boundary cohomology class vanishes. It could, however, happen that this boundary class is nonzero but still the last leading-term membership holds: one only needs the class to be in the image of the cup product $H^1\otimes H^1\to H^2$. As such, whether or not vanishing of the boundary is necessary depends only on the pairwise linking numbers of $L$, e.g., it is necessary when the linking numbers vanish in $R$ and unnecessary when they are all units.
\end{example}

\begin{remark}[Complete geometric obstruction]\label{rmk:geometric is sufficient}
    Heuristically, we can see that the existence of a complete surface system should be necessary as well as sufficient for membership in $\Primary(L)$. As transversality is a generic condition and Seifert surfaces always exist, the only obstruction for membership is equation \eqref{eq:vert diff = hors diff}: the intersections $\sum (\Sigma_j^\beta\cap \Sigma_{j+1}^\beta)$, which together form circles and arcs connecting link components, may or may not be a boundary in $M\setminus L$, and this sum bounds if and only if the number of outgoing arcs at each link component equals the number of incoming arcs.

    When the sum is homologically trivial, there already exists an immersed surface with that boundary and we may run the surface system algorithm described in \cref{rmk:algorithm} to construct a system of surfaces witnessing membership in $\Primary(L)$. We therefore expect that the existence of an arbitrary bar cocycle already implies existence of a complete surface system with the same leading term. We plan to formalize this in future work.
\end{remark}

\begin{proof}[Proof of \cref{prop:geometric criterion}]
    The input in this proof is the geometric cochain complex $C^*_\Gamma(M\setminus L;R)$; see \cref{A:geometric}. Every properly immersed, oriented  surface with corners $\Sigma \looparrowright M\setminus L$ represents a cochain, denoted $\Sigma\in C^*_\Gamma(M\setminus L;R)$ by abuse of notation. The bar construction on $C^*_\Gamma(M\setminus L;R)$ consists of tensors of these cochains in which all factors of each monomial are transverse (as sets of immersed surfaces).
    
    Suppose there exists a complete surface system with leading term $P(S_1,\ldots,S_n)$. For every tuple $(\Sigma_1^\alpha,\ldots,\Sigma_\ell^\alpha)$ of the surface system, construct a tensor of geometric cochains $[\Sigma_1^\alpha |\cdots |\Sigma_\ell^\alpha]\in \BBar_\Gamma^0(C^*_\Gamma(M\setminus L;R))$ (see \cref{A:geometric}). Then the equality in \eqref{eq:vert diff = hors diff} states that the horizontal bar differential $d_h$ cancels with the vertical bar differential $d_v$, together giving an element
    \begin{equation}\label{eq:P(S) for a surface system}
        \mathscr{P}(\Sigma) :=\sum_{\ell=1}^p (-1)^{p-\ell}\sum_{\alpha\in I_\ell}[\Sigma_1^\alpha |\cdots |\Sigma_\ell^\alpha] \in \BBar_\Gamma^0(C^*_\Gamma(M\setminus L;R))
    \end{equation}
    that is a cocycle for the total bar differential $d_{\BBar}$. Its image in the associated graded is exactly its leading term, so $P(S_1,\ldots,S_n)\in \Primary(L)$.

    For the converse statement, suppose an incomplete surface system exists and define $\mathscr{P}(\Sigma)\in \BBar^0$ as above, but without the $\ell=1$ summand, so that $d_{\BBar}\mathscr{P}(\Sigma) = (-1)^p\sum_{\alpha\in I_2}(\Sigma_1^\alpha\cap \Sigma_2^\alpha)$ represents the boundary of the surface system up to sign. Suppose there exists $T_{<p}\in \BBar_\Gamma^0(C^*_\Gamma(M\setminus L;R))$ of weight $<p$ with $d_{\BBar}(T_{<p})\in C^2_\Gamma(M\setminus L;R)$ cohomologous to $\sum_{\alpha\in I_2}(\Sigma_1^\alpha\cap \Sigma_2^\alpha)$ as cocycles in $C^2_\Gamma(M\setminus L;R)$, say with 
    \[
    \sum_{\alpha\in I_2}(\Sigma_1^\alpha\cap \Sigma_2^\alpha) - d_{\BBar}(T_{<p}) = d(\Sigma_1)
    \]
    for some geometric cochain $\Sigma_1\in C^1_\Gamma(M\setminus L;R)$. Then,
    \[
    \mathscr{P}(\Sigma) - (T_{<p} + \Sigma_1) \in \BBar^0_\Gamma(C^*_\Gamma(M\setminus L;R))
    \]
    is a cocycle with leading term $P(S_1,\ldots,S_n)$ since all added terms have strictly smaller weight. 
    
    For the ``only if'' direction, suppose that the leading term $P(S_1,\ldots,S_n)$ belongs to $\Primary(L)$, meaning that there exists some cocycle $T_p\in \BBar_{\Gamma}(C^*_\Gamma (M\setminus L;R))$ with leading term $P(S_1,\ldots,S_n)$. The weight spectral sequence for $\BBar_\Gamma(C_\Gamma^*(M\setminus L;R))$ (see \cref{prop:geometric weight spectral sequence}) has $E_1^{-p,p}\cong H^1(M\setminus L;R)^{\otimes p}$ and the projection of the difference $\mathscr{P}(\Sigma)-T_p$ to $E_1^{-p,p}$  represents the zero element.
    This implies that there exists some $T'\in \BBar^{-p,p-1}_\Gamma$ such that $d_{\BBar}(T') \equiv \mathscr{P}(\Sigma)-T_p$ modulo $\BBar_{<p}$, the submodule of terms of lower weight. 
    Then $T_{<p} := \mathscr{P}(\Sigma)-T_p -d_{Bar}(T')$ is an element of weight $<p$ satisfying
    \[
    d_{\BBar}(T_{<p}) = d_{Bar}(\mathscr{P}(\Sigma))-\cancel{d_{\BBar}(T_p)} - \cancel{d_{\BBar}^2(T')}
    \]
    as claimed to exist.\qedhere

\end{proof}

\subsection{Examples of the leading-term invariant via surface systems}

In the remainder of this section, we apply the geometric criterion and algorithm in examples.

\begin{example}[A boundary link]
    Let $L:nS^1\into M$ be an $n$-component boundary link, i.e., one for which it is possible to pick disjoint Seifert surfaces, such as the unlink in $S^3$. In constructing a surface system of these Seifert surfaces, no nontrivial intersections appear so such systems are already complete. It follows that $\Primary(L)=R\langle S_1,\ldots,S_n\rangle$. Correspondingly, all of Milnor's $\bar\mu$-invariants vanish in $R$; see \cref{prop:comparison with milnor}.
\end{example}

\begin{example}[Pairwise linking]\label{ex:pairwise}
    Let $L=(L_1,L_2)$ be the Hopf link in $S^3$, with two components exhibiting linking number $+1$. Any choice of Seifert surfaces $\Sigma_1$ and $\Sigma_2$ intersects at a collection of arcs that do not bound in $M\setminus L$. Thus, for example, $[S_1|S_2]$ and $[S_2|S_1]$ are excluded from $\Primary(L)$. On the other hand, the shuffle product $[S_1]\ast[S_2]=[S_1|S_2] + [S_2|S_1]$ is in $\Primary(L)$. In fact, $\Primary(L)$ is a (commutative) polynomial ring generated by $[S_1]$ and $[S_2]$ under the shuffle product. By the exclusion of the monomial $[S_1|S_2]$ for every coefficient ring $R$, the weight-two component $\Primary(L)_2$ detects that the pairwise linking number is $\pm 1$, see \cref{C:Milnor}.
\end{example}

\begin{example}[Triple linking]\label{ex:triple}
    Let $L=(L_1, L_2, L_3)$ be the Borromean rings, with three components such that no two are linked, but with triple linking number $\mu_{123}=1$. Any generic choice of Seifert surfaces $\Sigma_1,\Sigma_2,\Sigma_3$ will have the property that every two of them have null-homologous intersection: $\Sigma_i\cap \Sigma_j=\partial \Sigma_{ij}$. Constructing an incomplete surface system for the monomial $[S_1|S_2|S_3]$, one starts with
    \[
    \mathscr{P}(\Sigma) := [\Sigma_1|\Sigma_2|\Sigma_3] -\left( [\Sigma_{12}|\Sigma_3] + [\Sigma_{1}|\Sigma_{23}]\right),
    \]
    which can be made closed if and only if the $1$-cycle $(\Sigma_{12}\cap \Sigma_3)+(\Sigma_{1}\cap \Sigma_{23})$ bounds a properly immersed surface in $M\setminus L$. But this intersection in fact includes a single unpaired arc from $L_1$ to $L_3$, which represents a nontrivial element of $H^2(M\setminus L)$. No weight-$2$ bar cochain can reproduce the cohomology class of this arc as in \cref{prop:geometric criterion} because all pairwise linking numbers are $0$. Hence the surface system cannot be completed to a bar cocycle,
    reflecting that the triple linking number $\mu_{123}$ is $1$.  Equivalently, this is measured by the Massey triple product $\langle S_1,S_2,S_3\rangle\subseteq H^2(M\setminus L;R)$, as was first observed by Massey when he introduced his eponymous product. Thus the Borromean rings are detected by the weight-3 component $\Primary(L)_3$.
    
    In higher weights, the leading-term invariant $\Primary(L)$ includes a bar monomial only if no three distinct indices appear in succession, e.g., $[\cdots|S_2|S_1|S_3|\cdots]\notin \Primary(L)$, as iterated coproducts of such a monomial will include a tensor factor $(\ldots)\otimes [S_2|S_1|S_3]\otimes (\ldots)$, violating closure of $\Primary(L)$ under the coproduct.
    Some linear combinations of excluded monomials do belong to $\Primary(L)$ as shuffle products of allowed monomials; this reflects known relations among Milnor's $\mu$-invariants. For example, $[S_1|S_2]\ast[S_3]=[S_1|S_2|S_3] + [S_1|S_3|S_2]+[S_3|S_1|S_2]\in \Primary(L)$. 

    Here we can also see that our leading-term invariants will have special properties in addition to being Hopf algebras, coming from the fact that bar cohomology detects truncated group rings of all
    groups while link groups are a proper subclass of
    groups.  For example, when linking numbers vanish, symmetries of triple linking numbers imply that the leading-term invariant contains either all permutation of $[S_i | S_j | S_k]$ or none of them.

\end{example}

\begin{example}[Linking with a $3$-chain]
\label{ex:3-chain-with-4th}
    Consider the four-component link $L=(L_1,\dots,L_4)$ shown on the right in \cref{F:3-chain-with-4th-thru-middle}. Let $\Sigma_i$ denotes a planar Seifert surface for $L_i$ and pick $\Sigma_{123}$ to be the surface which is bounded by $(\Sigma_1\cap \Sigma_2) + (\Sigma_2 \cap \Sigma_3) + (\Sigma_3 \cap \Sigma_1)$ -- filling the hole in the middle of the $3$-chain, as illustrated in the middle of \cref{F:3-chain-with-4th-thru-middle}. Then $\Sigma_{123}\cap \Sigma_4$ consists of a single arc connecting $L_3$ to $L_4$, reflecting the linking of $L_4$ with the $3$-chain. This arc cannot be bounded by anything of smaller weight since $L_4$ does not link with any other component and the linking of the other components cannot contribute any intersection arcs that run to $L_4$. So the partial surface system discussed in \cref{ex:geometric cocycle for a 3-chain} does not extend to a full bar cocycle.
    This linking phenomenon corresponds to the exclusion
    \[
    \left([S_1|S_2]+[S_2|S_3]+[S_3|S_1]\right)|[S_4] =[S_1|S_2|S_4]+[S_2|S_3|S_4]+[S_3|S_1|S_4] \notin \Primary(L)_3
    \]
    in weight 3.
    
    If we consider the link $L'$ shown on the left in Figure \ref{F:3-chain-with-4th-thru-middle}, this analogous intersection $\Sigma_{123}\cap \Sigma_4$ is empty and thus the same polynomial belongs to $\Primary(L)$. Therefore, the leading-term invariant distinguishes $L$ and $L'$ at weight 3, even though $L$ and $L'$ have the same pairwise and triple linking numbers. 
    
    The total triple linking invariant of Davis, Nagel, Orson, and Powell \cite{DNOP20} also distinguishes these two links. Indeed, the clasp words for the green component differ between the two, with signed counts involving the purple component differing by 1, while all other clasp words agree.

\begin{figure}[h!]
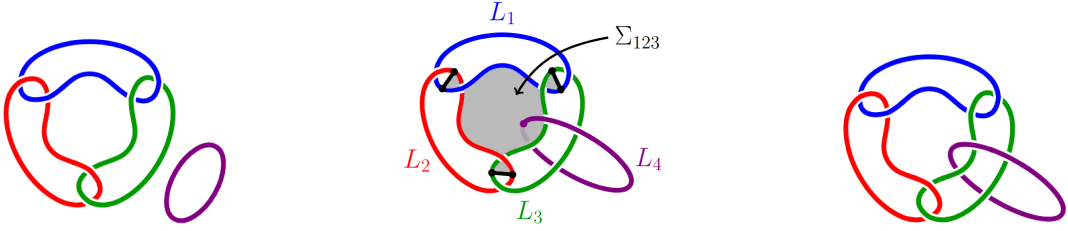

    \includegraphics[scale=0.5]{Tikz-Knots/Fig1_3-A.png} 
    \hspace{2cm}
    \includegraphics[scale=0.37]{Tikz-Knots/Fig2.png} 
    \hspace{2cm}
    \includegraphics[scale=0.5]{Tikz-Knots/Fig1_3-B.png} 
\caption{4-component links distinguished by our leading-term invariant at weight 3, but which are indistinguishable using classical Milnor invariants.}
\label{F:3-chain-with-4th-thru-middle}
\end{figure}
    \end{example}

While the DNOP total triple linking invariant
detects multi-component linking phenomena, as in \cref{ex:3-chain-with-4th}, their invariant applies to only what we classify as weight 3. Our invariant lets us consider all higher-weight linking, as in the following example of links that are not  distinguishable by the DNOP total triple linking invariant.

\begin{example}[A higher-weight example]
\label{ex:hopf-linked-3-chains}
    Consider the two 6-component links shown in \cref{F:hopf-linked-3-chains}. Call the link shown on the left-hand side $L = (L_1,\ldots,L_6)$ and the one on the right-hand side $L'=(L_1',\ldots,L_6')$.
    Let $\Sigma_i$ be a Seifert surface for $L_i$ or $L'_i$, and proceed as in \cref{ex:3-chain-with-4th} to consider for each link two surfaces $\Sigma_{123}$ and $\Sigma_{456}$ such that $\partial\Sigma_{ijk}=(\Sigma_i \cap \Sigma_j) +(\Sigma_j \cap \Sigma_k) +(\Sigma_k \cap \Sigma_i)$. The fact that the intersection $(\Sigma_{123} \cap \Sigma_{456})$ is empty on the left-hand side of \cref{F:hopf-linked-3-chains} implies that
    \begin{equation}
    \label{eq:excluded-polynomial}
    \left([S_1|S_2] + [S_2|S_3] + [S_3|S_1]\right) | \left([S_4|S_5] + [S_5|S_6] + [S_6|S_4]\right) = [S_1|S_2|S_4|S_5]+\ldots
    \end{equation}
    is included in $\Primary(L)$. In contrast, on the right-hand link $L'$ the same two surfaces intersect in an arc that connects the two 3-chains, which do not exhibit any pairwise- or triple-linking between them, and this implies that the same element in \eqref{eq:excluded-polynomial} is excluded from $\Primary(L')$.

    To construct an incomplete surface system with this leading term, one starts by extending the leading term $[S_i|S_j]+[S_j|S_k]+[S_k|S_i]$ to the bar cocycle
    \[
    [\Sigma_i|\Sigma_j]+[\Sigma_j|\Sigma_k]+[\Sigma_k|\Sigma_i] - [\Sigma_{ijk}].
    \]
    The tensor product of these for $(ijk)\in \{(123),(456)\}$ then consists of a sum of basic tensors represented by tuples of surfaces, most of which have empty intersections. Application of the geometric algorithm requires cobounding the nonempty intersections $\Sigma_{123}\cap \Sigma_4$ and $\Sigma_3\cap \Sigma_{456}$. Each of these consists of an arc whose ends both lie on the same component $L_i$ for $i\in \{3,4\}$, so half of the respective Seifert surface $\Sigma_i$ suffices and can be chosen to not intersect any further surfaces in the tuple. This gives an incomplete surface system with boundary $\Sigma_{123}\cap \Sigma_{456}$, which is empty in $L$ but in $L'$ it connects $L_3'$ to $L_4'$ and thus detects the linking as discussed above. Thus the leading-term invariant distinguishes $L$ and $L'$ at weight 4.
    
    
    It is easy to see that all classical Milnor invariants of weight at most 4 agree on this pair of links.  Moreover, their DNOP total triple-linking invariant is identical. To see this, note that the intersections $\Sigma_i \cap \Sigma_j$ are identical for all $i\neq j$ in the sense that the unions $(\Sigma_1\cup \cdots \cup \Sigma_6) \subseteq S^3$ for $L$ and for $L'$ are homeomorphic. These unions are called {\bf surface systems} by Davis, Nagel, Orson, and Powell, who show \cite[Theorem 1.1]{DNOP20} that the total triple-linking invariants of links admitting homeomorphic surface systems are equal. It also follows from that theorem that the respective link groups become isomorphic modulo the third step of their respective lower central series, so by the discussion in \cref{S:bounded truncations} $\Primary(L)_{\leq 3}=\Primary(L')_{\leq 3}$.

    
    In summary, the weight at which the distinction between $L$ and $L'$ appears is too high to be captured by the DNOP total triple linking invariant, and it does not show up in any of the classical Minor invariants due to their indeterminacy. On the other hand, since $L$ and $L'$ are satellite links, obtained from the links in \cref{ex:3-chain-with-4th} by replacing the fourth component by a chain, they can be distinguished by combining the result at weight 3 in 
    \cref{ex:3-chain-with-4th} with the link composition lemma \cite{Freedman-Lin:1989, Krushkal-Teichner:1997}.  Nonetheless, our invariant directly encodes this sharpening of Milnor invariants at arbitrarily high weights. 
    In addition, the next example -- \cref{ex:6-comp-sum-Wh-double-Borr} -- provides a non-satellite link in which our reasoning applies without change while the link composition lemma does not apply.
\end{example}

\begin{figure}[h!]
\includegraphics[scale=0.35]{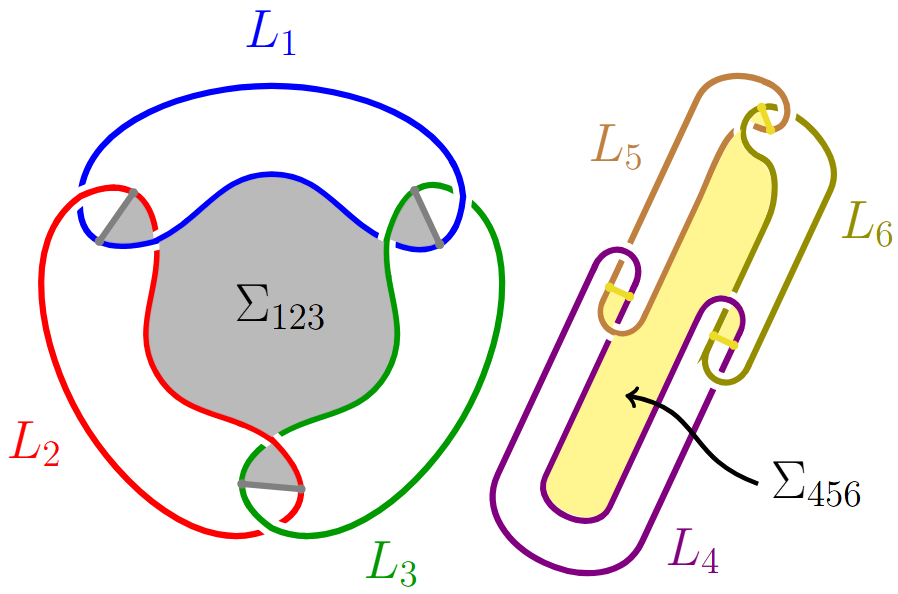}
\hspace{2cm}
\includegraphics[scale=0.35]{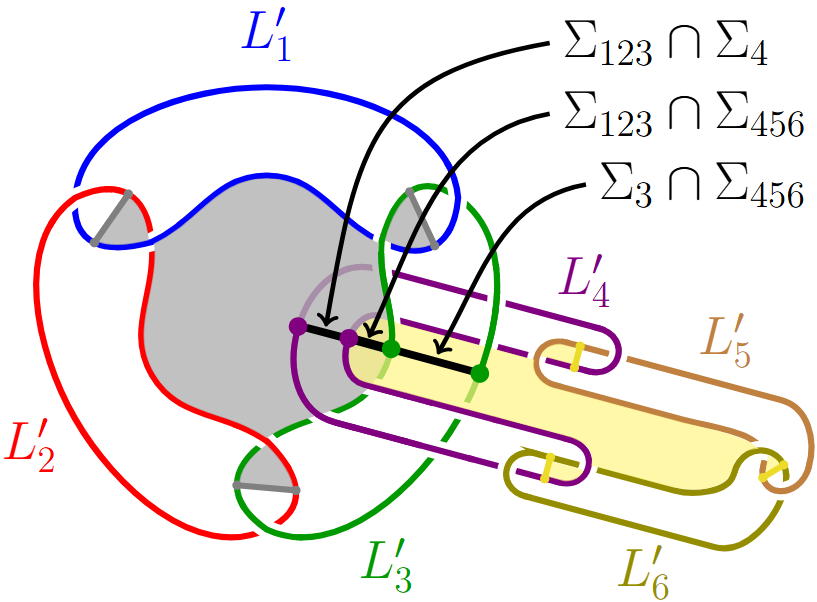}
\caption{A pair of 6-component links distinguished by our leading-term invariant $\Primary(-)$ at weight 4, but which is not distinguished by (classical) Milnor invariants.}
\label{F:hopf-linked-3-chains}
\end{figure}

\begin{example}[Non-satellite example]
\label{ex:6-comp-sum-Wh-double-Borr}
Let us provide an example that is not a satellite and therefore cannot be studied using recursive techniques such as the link composition lemma \cite{Freedman-Lin:1989, Krushkal-Teichner:1997}. We consider a variation $L''$ on the link $L'$ on the right of \cref{F:hopf-linked-3-chains}, shown in \cref{F:6-comp-sum}.  We obtain $L''$ by taking a band sum of $L'$ with the Whitehead double of the Borromean rings so that components from both 3-chains are involved. 

We distinguish the link $L''$ from the split link $L$ on the left-hand side of \cref{F:hopf-linked-3-chains} exactly like we distinguished $L$ and $L'$ in \cref{ex:hopf-linked-3-chains}: the same incomplete surface system cannot be completed in $L''$ due to the intersection $\Sigma_{123}\cap \Sigma_{456}$ consisting of an arc connecting the two chains. This highlights a strength of our invariant in that it identifies calculations that can be performed locally -- within a ball, while safely ignoring its exterior -- while still yielding well-defined invariants.

The link $L''$ is at least not obviously a satellite link, and in fact, the output of SnapPy indicates that it is almost certainly hyperbolic.  If true, this would imply it is not a satellite and hence the link composition lemma cannot be applied.  Moreover, the Whitehead double of any link with vanishing pairwise linking numbers is a boundary link, hence indistinguishable from the unlink using Milnor invariants, the total triple linking number, and our bar invariant.
Therefore, $L''$ cannot be distinguished from $L$ by any lower-weight terms or, to the best of our knowledge, by any combination of previously known refinements of Milnor invariants.  

\begin{figure}[h!]
\includegraphics[scale=0.3,angle=-90]{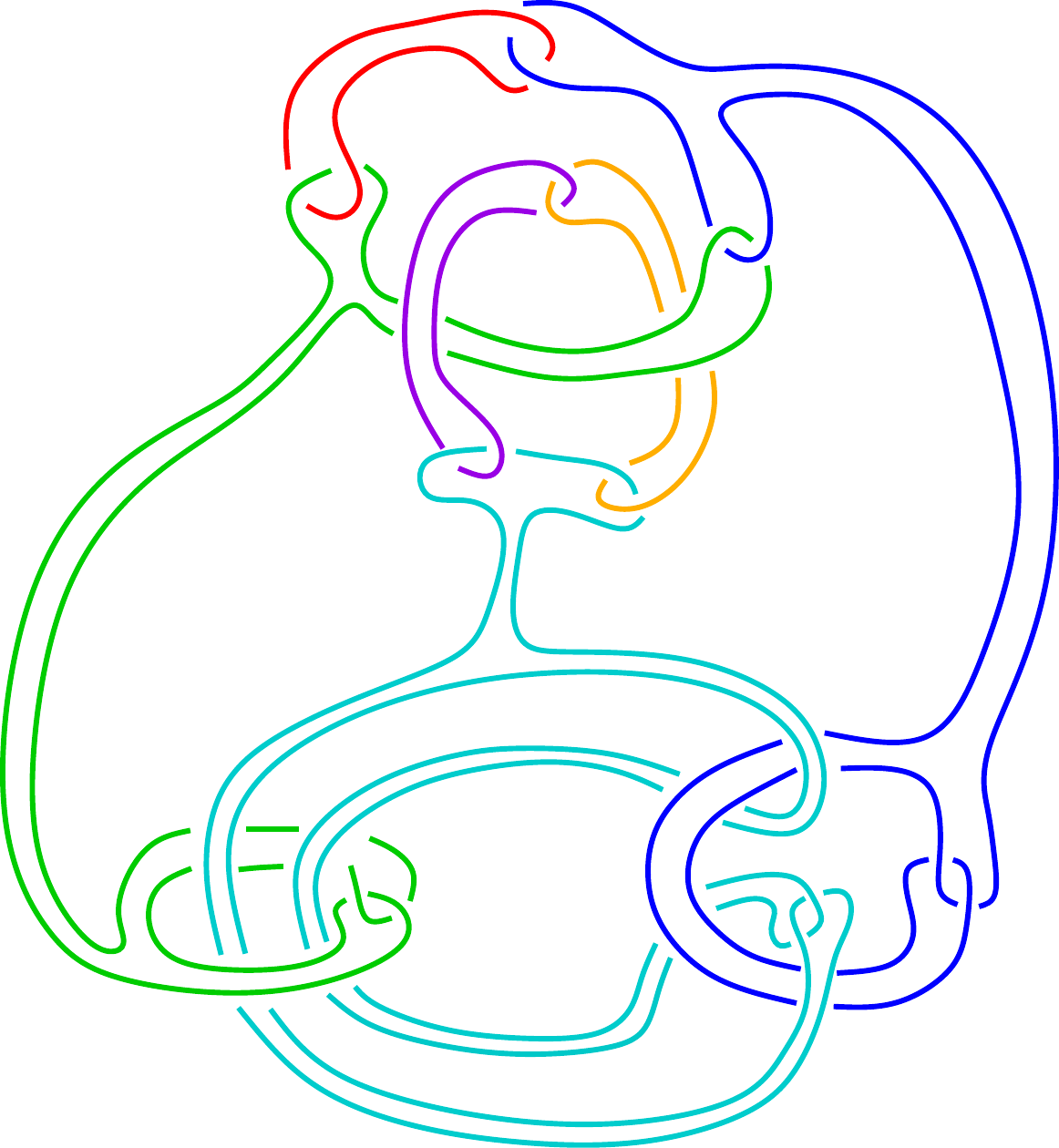}
\caption{A band sum $L''$ of the 6-component Hopf-linked 3-chains $L'$ from the right of \cref{F:hopf-linked-3-chains} with the Whitehead double of the Borromean rings.  The leading-term invariant $\Primary(-)$ distinguishes $L''$ from the disjoint union $L$ of two 3-chains at weight $4$, and $L''$ is (almost certainly) not a satellite link, so previously known refinements of Milnor invariants do not distinguish $L''$ and $L$.}
\label{F:6-comp-sum}
\end{figure}


To find a pair of non-split links that is distinguished in a similar way by our invariant, one could replace the split link $L$ by a link $L'''$ analogous to $L'$ but with the role of the Hopf link played by the Whitehead link.  Then our invariant distinguishes both $L'$ and $L''$ from $L'''$.  If one wants a pair where both links are neither split nor satellites, one could take $L''$ and the result of band-summing $L'''$ with the Whitehead double of the Borromean rings.  For any of these pairs, the argument that our invariant distinguishes them is similar to the one above, so we omit further details.

\end{example}

\section{Recovering Milnor invariants}\label{S: Milnor}
In this section, we show that
for a link $L:nS^1\into  S^3$,
the leading-term invariant $\Primary(L)$ over $R = \Z_\Lambda$ 
detects the nonvanishing of Milnor $\bar\mu$-invariants in $R$.
In particular, letting the modulus $\Lambda$ vary, the leading-term invariants compute all $\mu$-invariants up to units in the appropriate quotient ring of $\Z$.
Thus they generalize Cochran's interpretation \cite{Co90}  of these invariants in terms of intersections of Seifert surfaces.  

First, we recall Milnor's definition, as given in detail in Meilhan's survey \cite{Me18}. Let $L:nS^1\into  S^3$ be an $n$-component link and $*\in S^3\setminus L$ a basepoint. Pick pointed meridians $\mathpzc{m}_1, \dots, \mathpzc{m}_n: S^1\to S^3\setminus L$ where $\mathpzc{m}_i$ has linking number 1 with the $i$-th link component and does not link the other components. Abbreviate $\pi := \pi_1(S^3\setminus L,*)$ and denote its lower central series by $\gamma_1\pi :=\pi$ and $\gamma_{p+1}\pi := [\gamma_p\pi,\pi]$.  Then Milnor  \cite[Theorem 4]{Mi57b} observed that for any $p$, the nilpotent quotient $\pi/\gamma_{p+1} \pi$ is generated by the images of $(\mathpzc{m}_1,\ldots, \mathpzc{m}_n)$.


In particular, letting $\ell_i$ denote the homotopy class of the longitude of the $i$-th component, for every fixed $p$ it is possible to express each $\ell_i$ as a word in the generators.  That is, there are elements $w_1,\ldots,w_n\in F_n$ in the free group such that $\ell_i \equiv w_i(\mathpzc{m}_1,\ldots,\mathpzc{m}_n) \mod \gamma_{p+1}\pi$, where $w(\mathpzc{m}_1,\dots,\mathpzc{m}_n)$ is the image of $w_i$ under the homomorphism $F_n\to \pi$ sending the $i$-th generator to $\mathpzc{m}_i$. Milnor's $\mu$ invariants are then defined as the coefficients of these words under the Magnus--Fox expansion \cite{Me18} \cite[Ch.~12]{CDM}, which expresses elements $w\in F_n$ as a Taylor polynomial.  That is, if 
$I := (x_1-1,\ldots,x_n-1)$ is the augmentation ideal in the group ring $\Z[F_n]$ and $p$ is any natural number, these are the unique expansions, modulo $I^{p+1}$,
\[
w_{i_0} = 1+\sum_{\substack{i_1,\dots ,i_r\\ r\leq p}} \mu_{i_0i_1\cdots i_r} \cdot (x_{i_1}-1)(x_{i_2}-1)\cdots (x_{i_r}-1) + I^{p+1} \in \Z[F_n]/I^{p+1}.
\]
where $x_1,\ldots,x_n$ are the generators of $F_n$, and the sum is over all ordered tuples $(i_1,\dots,i_r)$ with $r \leq p$.

For any multi-index $I=(i_0, \dots, i_p)$ with  $1\leq i_0, \dots, i_p \leq n$, the integer $\mu_{i_0\cdots i_p}:=\mu_{i_0\cdots i_p}(L)$ is well-defined modulo the greatest common divisor $\Delta:=\Delta_{i_0\cdots i_p}(L)$ 
of the numbers $\mu_{j_0\cdots j_r}(L)$ over all multi-indices $J=(j_0, \dots, j_r)$ that are cyclic permutations of proper subsequences of $(i_0,\dots, i_p)$.
The resulting residues are Milnor's higher-order linking numbers, often called the $\bar\mu$-invariants.
We call $\mu_{i_0\cdots i_p}$ a {\bf Milnor invariant of  weight $p+1$}.
Examples include pairwise linking numbers and triple linking numbers.
In particular, for any modulus $\Lambda \in \Z$ (including $\Lambda=0$) and any multi-index $I$, if 
$\mu_{J}(L)\equiv0 \mod \Lambda$ for all cyclic permutations $J$ of proper subsequences of $I$, then the integer $\mu_{I}(L)$ modulo $\Lambda$ is independent of all choices made.

The $\bar\mu$-invariants satisfy two useful sets of identities \cite[\S3.3]{Me18}.  The first is cyclic symmetry \cite[Theorem 6]{Mi57b}, which consists of the relations $\mu_{i_0\cdots i_p} \equiv \mu_{i_1\cdots i_p i_0} \mod \Delta$.  The second are the shuffle relations \cite[Equation (23)]{Mi57b}.
The shuffle relations imply that any nontrivial Milnor invariant has at least two distinct indices, and cyclic symmetry implies that any such invariant is equal to one where the first and last indices are distinct.  The latter property appears in \cref{prop:comparison with milnor} and \cref{C:Milnor} below.

Recall that for an $n$-component link $L$, 
the leading-term invariant $\Primary(L; R)$ with coefficients in $R$ is a graded submodule of the cofree Hopf algebra $R\langle S_1,\ldots,S_n\rangle$, where $S_i$ is the cohomology class of a Seifert surface for the $i$-th component of $L$. In particular, one can ask whether a particular monomial in the $S_i$'s belongs to $\Primary(L; R)$.

\begin{theorem}
\label{prop:comparison with milnor}
    Let $L$ be an $n$-component link in $S^3$.
    Fix a modulus $\Lambda\in \Z$, 
    and fix a multi-index $(i_0,\ldots i_p)$ where $1\leq i_0, \dots, i_p \leq n$. Assuming $\mu_{j_0\cdots j_r}(L) \equiv 0 \mod \Lambda$ for all cyclic permutations $(j_0,\ldots,j_r)$ of proper subsequences of $(i_0,\ldots i_p)$, we have that if $[S_{i_0}|S_{i_1}|\cdots |S_{i_p}]\not\in \Primary(L;\Z_\Lambda)$, then $\mu_{i_0\cdots i_p}(L)\not\equiv0 \mod \Lambda$ and $i_0\neq i_p$.

    Assuming $\mu_{j_0\cdots j_r}(L)\equiv 0\mod \Lambda$ for every multi-index 
    $(j_0,\ldots,j_r)$ with $r<p$, we have that $\sum a_{i_0\cdots i_p}[S_{i_0}|S_{i_1}|\cdots |S_{i_p}]\in \Primary(L;\Z_\Lambda)$ if and only if 
    \begin{equation}
    \label{eq:Kronecker deltas thm}
        \sum_{i_0,\dots,i_p} (\delta_{s,i_p}-\delta_{s,i_0})a_{i_0\cdots i_p}\cdot \mu_{i_0\cdots i_p}(L)\equiv 0 \mod \Lambda
    \end{equation}
    for every $s\in \{1, \dots, n\}$, where $\delta_{i,j}$ is the Kronecker delta.
\end{theorem}

The statement about linear combinations of monomials immediately gives a partial converse to the statement about a single monomial: if $i_0\neq i_p$ and $\mu_{j_0\cdots j_r}(L)\equiv 0 \mod \Lambda$ for all multi-indices with $(j_0,\dots,j_r)$ with $r<p$, then 
$[S_{i_0}|\dots|S_{i_p}] \in \Primary(L;\Z_\Lambda)$ implies $\mu_{i_0\cdots i_p}(L) \equiv 0 \mod \Lambda$.
Note also that the congruences \eqref{eq:Kronecker deltas thm} are equivalent to congruences 
\[
\sum_{i_0,\ldots,i_{p-1}} a_{i_0\cdots i_{p-1}s}\cdot \mu_{i_0\cdots i_{p-1}s}(L) \equiv \sum_{i_1,\ldots,i_{p}} a_{si_1\cdots i_p}\cdot \mu_{si_1\cdots i_{p}}(L) \mod \Lambda
\]
for every $s \in \{1,\dots,n\}$.

\begin{proof}
    The connection between the bar construction and Milnor's invariants goes through Massey products: 
    \[
    \{ \bar\mu\text{-invariants}\} \xleftrightarrow{\text{Turaev, Porter}} \{\text{Massey products}\} \xleftrightarrow{\text{May}} \{\text{bar construction}\}.
    \]
    Indeed, May considered the bar spectral sequence, and he showed \cite[Theorem 6 and Lemma 4]{May66} that for any coefficient ring $R$, if $S_{i_0},\ldots,S_{i_p}\in C^*(X;R)$ are cochains for which the $(p+1)$-st Massey product is defined and contains the class of $\beta \in C^*(X;R)$, then $[S_{i_0}|\ldots|S_{i_p}]\in \BBar(C^*(X;R))$ survives to the $p$-th page of the weight spectral sequence and $d_p[S_{i_0}|\ldots|S_{i_p}] =[\beta]$.

    To address the connection to Milnor's $\bar\mu$-invariants, we set $\Delta:=\Delta_{i_0,\ldots,i_p}$ to be the greatest common divisor of $\mu_{j_0\cdots j_r}$ for cyclic permutations of proper subsequences of $(i_0,\ldots,i_r)$.  We also recall that $H^2(S^3\setminus L;\Z) = \langle \tau_{ij} := \tau_i-\tau_j \mid i\neq j\rangle\subseteq \langle \tau_1,\ldots,\tau_n\rangle $ is generated by classes $\tau_{ij}$ Lefschetz dual to a path running from the $i$-th link component to the $j$-th link component; such paths represent these elements exactly in geometric cohomology. Porter \cite[Theorem 3]{Po80} showed that the Massey product $\langle S_{i_0},\cdots,S_{i_p}\rangle$ is defined in cohomology with coefficients in $\Z_\Delta$
    and contains a unique element $(-1)^{p+1} \mu_{i_0\cdots i_p}\cdot  \tau_{i_0i_p}\in H^2(S^3\setminus L; \Z_\Delta)$.  A similar result was obtained earlier and independently by Turaev \cite{Tu76}.

    Our assumption that $\mu_{j_0\cdots j_r} \equiv 0 \mod \Lambda$ for all appropriate sequences $(j_0,\ldots,j_r)$ implies that $\Lambda$ divides $\Delta$.
    Applying reduction mod $\Lambda$, the Porter--Turaev result shows that the Massey product $\langle S_{i_0},\cdots,S_{i_p}\rangle$ is defined modulo $\Lambda$ and contains the element $(-1)^{p+1} \mu_{i_0\cdots i_p}\cdot  \tau_{i_0i_p}\in H^2(S^3\setminus L;\Z_\Lambda)$. Together with May's calculation it follows that $[S_{i_0}|\ldots|S_{i_p}]$ survives to the $p$-th page of the bar spectral sequence and $d_p:[S_{i_0}|\ldots|S_{i_p}]\mapsto (-1)^{p+1} \mu_{i_0\cdots i_p}\cdot  \tau_{i_0i_p} \mod \Lambda$.
    Since $d_p$ is the last differential that could kill $[S_{i_0}|\ldots|S_{i_p}]$, a necessary condition for $[S_{i_0}|\ldots|S_{i_p}]\notin E^\infty = \Primary(L; \Z_\Lambda)$ is that $\mu_{i_0\cdots i_p}\cdot  \tau_{i_0i_p} \neq 0$ in $H^2(S^3\setminus L; \Z_\Lambda)$, which is equivalent to having $i_0\neq i_p$ and $\mu_{i_0\cdots i_p} \not\equiv 0 \mod \Lambda$.  (It is not a sufficient condition because $\mu_{i_0\cdots i_p}\tau_{i_0i_p}$ could be killed by an earlier differential.)
    
    For the statement about linear combinations of monomials, we assume that $\mu_{j_0\cdots j_r}=0 \mod \Lambda$ for all tuples $(j_0,\dots,j_r)$ of length $\leq p$. The argument in the previous paragraph shows that any monomial $[S_{j_0}|\cdots |S_{j_r}]$ with $r<p$ survives to the $E^r$-page of the bar spectral sequence, and there $d_r[S_{j_0}|\cdots |S_{j_r}] = (-1)^{r+1}\mu_{j_0\cdots j_r}\cdot \tau_{j_0 j_r} = 0$ by the vanishing assumption on the Milnor numbers. Thus, all differentials $d_r:E_r^{-{r+1,r+1}}\to E_r^{-1,2}\cong H^2(S^3\setminus L;\Z_\Lambda)$ with $r<p$ are identically zero. It follows that on the $E^p$-page, we have $E_p^{-1,2}\cong E_1^{-1,2} \cong H^2(S^3\setminus L;\Z_\Lambda)$. By the same May--Porter--Turaev argument again, the $p$-th differential is given by 
    \[
    d_p:\sum a_{i_0\cdots i_p}[S_{i_0}|\ldots|S_{i_p}] \longmapsto (-1)^{p+1}\sum a_{i_0\cdots i_p}\cdot \mu_{i_0\cdots i_p}\cdot  \tau_{i_0i_p}\in H^2(S^3\setminus L;\Z_\Lambda),
    \]
    so $\sum a_{i_0\cdots i_p}[S_{i_0}|\ldots|S_{i_p}]\in \Primary(L; \Z_\Lambda)$ if and only if $\sum a_{i_0\cdots i_p}\cdot \mu_{i_0\cdots i_p}\cdot  \tau_{i_0i_p} = 0$ in cohomology, which is equivalent to the vanishing stated in the theorem
    by evaluating this cohomology class on small tori around the link components.
\end{proof}

The first nontrivial case of Equation \eqref{eq:Kronecker deltas thm} is given in the next example. A different general implication is given in Corollary \ref{C:Milnor} below.

\begin{example}[Three equal linking numbers]
\label{ex:same linking numbers}
    Consider a 3-component link $L:nS^1\into  S^3$ in which 
    all pairwise linking numbers are equal, i.e., $\mu_{ij}=s$, $\forall i\neq j$. For example, the three linked components in \cref{ex:3-chain-with-4th}.    
    Then over the integers $(\Lambda =0)$ we have
    \[ [S_1|S_2]+[S_2|S_3]+[S_3|S_1]\in \Primary(L; \Z)\]
    since $(\delta_{s,2}-\delta_{s,1})\mu_{12} + (\delta_{s,3}-\delta_{s,2})\mu_{23} + (\delta_{s,1}-\delta_{s,3})\mu_{31} = \mu_{(s-1)s} -\mu_{s(s+1)} = 0$ for all $s=1,2,3$, where arithmetic of indices is computed modulo three. In this sense, membership in $\Primary(L; \Z)$ reflects the equality of the three linking numbers. In \cref{S:examples} we used a cocycle ``witness'' of this equality to build finer invariants when such links $L$ occur within larger links; see e.g.\ \cref{ex:hopf-linked-3-chains}.  For $\Lambda\neq 0$, one can equally well use $\Primary(L; \Z_\Lambda)$ to detect when the three linking numbers are congruent modulo $\Lambda$.
\end{example}

\begin{corollary}\label{C:Milnor}
    Fix a link $L:nS^1\into  S^3$
    and a multi-index $I=(i_0,\dots,i_p)$. Let $\Delta$ denote the greatest common divisor of all $\mu_{j_0\cdots  j_r}(L)$ over cyclic permutations $(j_0,\dots,j_r)$ of proper subsequences of $I$.  Then the collection of submodules $\Primary(L; \Z_\Lambda)\leq \Z_\Lambda \langle S_1,\ldots,S_n\rangle$ 
    over $\Lambda \in \Z$
    determines $\Delta$, as well as the residue class of the $\bar\mu$-invariant $\mu_{i_0\ldots i_p} \mod \Delta$ up to multiplication by units in $(\Z_\Delta)^\times$.   
    
    If $\Lambda$ is the greatest common divisor of all $\bar\mu$-invariants of $L$ of length $\leq p$, then the weight-$(p+1)$ submodule $\Primary(L; \Z_\Lambda)_{p+1}\subseteq \Z_\Lambda\langle S_1,\ldots,S_n\rangle_{p+1}$ detects all relations modulo $\Lambda$ between $\bar \mu$-invariants of length $(p+1)$ for which the extremal indices are fixed and distinct. That is, for any fixed pair $i_0\neq i_p$,
    \begin{equation}
    \label{eq:fixed i_0 i_p}
        \sum_{i_1\ldots i_{p-1}} a_{i_0\cdots i_p}\cdot \mu_{i_0\cdots i_p}(L) \equiv 0 \mod \Lambda
    \end{equation}
    if and only if 
    $\sum a_{i_0\cdots i_p}\cdot [S_{i_0}|S_{i_1}|\cdots | S_{i_p}]\in \Primary(L; \Z_\Lambda)$. In particular, when $\Lambda\neq 0$ the set of coefficients $a\in \Z_\Lambda$ for which $a\cdot [S_{i_0}|S_{i_1}|\cdots | S_{i_p}]\in \Primary(L;\Z_\Lambda)$ already determines $\mu_{i_0\cdots i_p}\in \Z_\Lambda$ up to units $(\Z_\Lambda)^\times$.
\end{corollary}

\begin{proof}
    By \cref{prop:comparison with milnor}, $\Delta$ as defined in the corollary statement is the largest integer $\Lambda$ for which $\Primary(L;\Z_\Lambda)$
    contains all monomials $[S_{j_0}| \dots |S_{j_r}]$ for $(j_0,\dots,j_r)$ that are cyclic permutations of proper subsequences of $I$.
    
    We next check that the collection of submodules $\Primary(L; \Z_\Lambda)\leq \Z_\Lambda\langle S_1,\ldots,S_n\rangle$ over all $\Lambda$ dividing $\Delta$ determines the residue $\mu_{i_0\cdots i_p}(L) \mod \Delta$, up to units in $\Z_\Delta$.  Indeed, by \cref{prop:comparison with milnor}, if $i_0\neq i_p$ and $[S_{i_0}|\ldots|S_{i_p}]\in \Primary(L; \Z_\Lambda)$, then $\Lambda$ divides $\mu_{i_0\cdots i_p}$ in $\Z_\Delta$. Thus the largest $\Lambda$ dividing $\Delta$ for which this monomial belongs to $\Primary(L; \Z_\Lambda)$ satisfies $\Lambda \equiv \alpha \mu_{i_0\cdots i_p} \mod \Delta$ for some $\alpha$ relatively prime to $\Delta$, i.e., for some $\alpha \in (\Z_\Delta)^\times$.

    Now let $\Lambda$ be the greatest common divisor of all $\mu_{j_0\cdots j_r}$ with $r<p$.  
    Specializing Theorem \ref{prop:comparison with milnor} to a linear combination of invariants with fixed extremal indices, 
    we have that 
    \[
    \sum_{i_1, \dots, i_{p-1}} a_{i_0\cdots i_p}\cdot [S_{i_0}|S_{i_1}|\cdots | S_{i_p}]\in \Primary(L; \Z_\Lambda) 
    \]
    if and only if the congruence \eqref{eq:Kronecker deltas thm} holds for every $s\in\{1,\dots,n\}$.  
    Since $i_0$ and $i_p$ are fixed, each value of $s$ corresponds to either the desired congruence \eqref{eq:fixed i_0 i_p} or the trivial relation.

    Lastly, when $\Lambda\neq 0$ every element $\mu\in \Z_\Lambda$ is determined up to units by its set of zero divisors. Therefore, the set of coefficients $a\in \Z_\Lambda$ for which $ a\cdot [S_{i_0}|S_{i_1}|\cdots | S_{i_p}]\in \Primary(L; \Z_\Lambda) $, which coincides with the set of zero divisors of $\mu_{i_0\cdots i_p}\in \Z_\Lambda$, determines the latter up to a unit.
\end{proof}


\begin{example}[Milnor invariants up to units]

To illustrate how Corollary \ref{C:Milnor} detects the values of Milnor invariants up to an appropriate unit, consider a link $L$ for which all Milnor invariants of weight $\leq p$ are divisible by $8$. This implies that each weight-$(p+1)$ invariant is defined mod 8. For any such $\mu_{i_0\cdots i_p}\in \Z_8$ with $i_0\neq i_p$, the corresponding monomial $[S_{i_0}|\cdots |S_{i_p}]$ is in $\Primary(L;\Z_8)$ if and only if $\mu_{i_0\cdots i_p} \equiv 0 \mod 8$. If $[S_{i_0}|\cdots |S_{i_p}]$ is not in $\Primary(L;\Z_8)$ but is in $\Primary(L;\Z_4)$, then $\mu_{i_0\cdots i_p}$ is not 0 mod 8, but it is 0 mod 4, so it must be 4 mod 8.  Similarly, if $[S_{i_0}|\cdots |S_{i_p}]$ is not in $\Primary(L;\Z_4)$ but is in $\Primary(L;\Z_2)$, then $\mu_{i_0\cdots i_p}$ is not 0 mod 4, but it is 0 mod 2, so it must be either 2 or 6 mod 8. Note that 2 and 6 differ by a unit in $\Z_8$, so 2 and 6 generate the same ideal of $\Z_8$ and hence the same quotient. Finally, if $[S_{i_0}|\cdots |S_{i_p}]\notin \Primary(L;\Z_2)$, then $\mu_{i_0\cdots i_p}$ must be odd, but we cannot detect whether it is $1$, $3$, $5$, or $7$ mod 8, as reducing the coefficient ring further by any of these kills the coefficient ring completely.  In summary, by this process we can determine to which of the sets $\{0\}$, $\{4\}$, $\{2,6\}$, or $\{1,3,5,7\}$ the Milnor invariant $\mu_{i_0\cdots i_p}$ belongs mod 8. 

More directly, if $2^k[S_{i_0}|\cdots |S_{i_p}]$ is in $\Primary(L;\Z_8)$ but $2^{k-1}[S_{i_0}|\cdots |S_{i_p}]$ is not then $\bar\mu_{i_0\cdots i_p}$ must be $8/2^k$ up to units.
\end{example}

\appendix
\section{The bar construction on geometric cochains}\label{A:geometric}

Geometric cochains are an alternative to singular cochains that give geometric representatives to cohomology classes and make the intersection theory manifest at the cochain level. They form a cochain complex\footnote{The symbol $\Gamma$ was chosen in \cite{FMS-foundations} to stand for ``geometric.'' The close relation between Segal's $\Gamma$ structures and the homotopy algebra structure on $C^*_\Gamma(M)$, which we describe in this appendix, is an unintentional happy coincidence of notation but should not be assigned any meaning.} $C^*_\Gamma(M)$ associated to a smooth manifold $M$ whose elements are equivalence classes of proper oriented\footnote{\label{F:co-orient} In general, geometric cochains are represented by co-oriented maps, but when the codomain is oriented, co-orientations are equivalent to orientations of the domain. As $R$-homology spheres are always orientable, we never need to consider co-orientations in this paper. We refer to \cite{FMS-foundations} for orientation conventions for boundaries and fiber products, noting that in the cochain case we must work first with the co-orientation conventions and then convert to the orientations induced by the orientation of codomain.} smooth maps from manifolds with corners to $M$; see \cite{FMS-foundations} for more detailed background. Their cohomology is isomorphic to singular cohomology. 

One advantage of this cohomology theory is that the cup product, represented by the fiber product of transverse maps, is strictly (graded) commutative. These products are only  partially defined on cochains, due to the transversality constraint, though they are fully defined on cohomology classes, as it is always possible to find transverse cocycles representing the cohomology classes. In case the cochains are represented by transverse embeddings (or immersions), fiber product is represented by the embedding (or immersion) of the intersection, properly oriented.
The isomorphism with singular cohomology sends the cup product to the usual cohomology cup product, realizing the slogan that cup product is dual to intersection on appropriate representatives.

We apply the bar construction to the geometric cochain complex of a link complement.
Bar cocycles are constructed by picking surfaces, intersecting them, bounding the intersections, and iterating as necessary, as codified in Section~\ref{surface systems}. Here we explain what this bar construction is and compare it to the standard bar construction discussed in \cref{S: bar cohomology}.

\subsection{Homotopy algebras}

As shown in \cite{GBF47}, the product of geometric cochains is part of a structure identified by Leinster \cite{leinster2000homotopy}, that of a homotopy commutative monoid in chain complexes. McClure refers to such structures as Leinster partial algebras in \cite{McC06}.

\begin{definition}
Let $\mathrm{Fin}$ denote the category of finite sets, with skeletal subcategory given by the finite cardinals $[p] = \{1,\ldots,p\}$ starting with $[0]=\emptyset$. This is a symmetric monoidal category under the disjoint union operation $\oplus := \coprod$
\[
\oplus :\mathrm{Fin}\times \mathrm{Fin}\to \mathrm{Fin}, \quad\text{satisfying}\quad [p]\oplus [q] \cong [p+q].
\]
If $(\mathrm{C},\otimes ,\mathbb{I})$ is a symmetric monoidal category with a class of weak equivalences satisfying a set of natural axioms  (see \cite[Definition 2.1.1]{leinster2000homotopy}), a {\bf homotopy commutative algebra} in $\mathrm{C}$ is a colax symmetric monoidal functor $A^\bullet:(\mathrm{Fin},\oplus,[0])\to (\mathrm{C},\otimes ,\mathbb{I})$ such that the `colax' structure maps satisfy {\bf the Leinster condition}: 
\[
\eta_{S,T}:A(S\oplus T) \to A(S)\otimes A(T) \text{ and } \eta_0: A([0]) \to \mathbb{I}
\]
are equivalences. Denote $A^{[p]}$ for the value $A([p])$ and call the value $A^{[1]}$ the {\bf underlying object} of the algebra. It is common to abuse notation and write $A:= A^{[1]}$.
\end{definition}
Note that since $[p] = [1]^{\oplus p}$, there are equivalences $A^{[p]}\to A^{\otimes p}$ that are equivariant with respect to the symmetric group action on both sides.  McClure \cite{McC06} considers the special case where all $A^{[p]}\to A^{\otimes p}$ are inclusions and refers to them as \textbf{domains} (of definition) for the product.

One thinks of a homotopy commutative algebra as a multiplication on the underlying object $A$, with product given by the zig-zag
\[
A^{\otimes 2} \xleftarrow{\sim} A^{[2]}\xrightarrow{\nabla_*} A,
\]
where $\nabla_*$ is the map induced by $\nabla: [2]\to [1]$ sending both nontrivial elements of $[2]$ to the nontrivial element of $[1]$. Note that both maps in this zig-zag are invariant under the map that swaps $1\leftrightarrow 2$, alluding to a strictly commutative structure, but having to invert the first equivalence breaks the symmetry and makes this property only hold up to homotopy.

Conversely, a commutative algebra $A$ defines a homotopy commutative algebra by $A^{[p]}:= A^{\otimes p}$, where the structure maps $[p]\to [q]$ induce $A^{\otimes p}\to A^{\otimes q}$ that multiply all elements in a fiber. If $[p]\to [q]$ is not surjective, we place the unit of $A$ in the tensor factors corresponding to indices not in the image.  

\begin{remark}
    Leinster explicitly excludes homotopy algebras in the category of cochain complexes and quasi-isomorphisms from his list of examples. This is because quasi-isomorphisms are not preserved by a tensor with a fixed complex. However, restricting to the subcategory of bounded-below complexes of flat $R$-modules, in which geometric cochains reside, the K\"unneth spectral sequence shows that quasi-isomorphisms are respected by the tensor product and Leinster's axioms hold.
\end{remark}

\begin{example}
    If $M$ is a smooth manifold without boundary, geometric cochains $C^*_\Gamma(M)$, whose elements are represented by proper oriented maps of manifolds with corners $N\to M$, can be given the structure of a homotopy commutative algebra with $A^{[p]}$ defined essentially as  
    \[
    \textrm{Span}\{N_1\otimes \cdots \otimes N_p \mid \text{every subset of these manifolds meets transversally\footnotemark }\} \subseteq C^*_\Gamma(M)^{\otimes p},
    \]
    and where a map $f:[p]\to [q]$ in $\mathrm{Fin}$ sends the tensor of geometric cochains $N_1\otimes \cdots\otimes  N_p$ to the tensor
    \[
    \pm\left(\underset{j\in f^{-1}(1)}{\olduplus} N_j\right) \otimes \left(\underset{j\in f^{-1}(2)}{\olduplus} N_j\right)\otimes   \cdots \left(\otimes \underset{j\in f^{-1}(q)}{\olduplus} N_j\right),
    \]
    with an overall sign following the Koszul sign convention, i.e. having the simple transposition $(i\leftrightarrow i+1)$ act with sign $(-1)^{|N_i||N_{i+1}|}$.
    Here $\olduplus$ denotes the fiber product of transverse geometric cochains, which is the cup product of geometric cochains, and if $f^{-1}(k)$ is empty, we place the identity map $M \to M$ in that factor.  
    The Leinster condition of homotopy equivalence amounts to the observation that the inclusion $A^{[p]}\into C^*_\Gamma(M)^{\otimes p}$ is a quasi-isomorphism, as every tensor of cohomology classes has mutually transverse representatives. See \cite{GBF47} for details.
\end{example}
\footnotetext{Here we mean transversality of a set of maps in the strong sense that not only must each pair be transverse, but each intersection (fiber product) of a pair must be transverse to all other elements, and so on for all possible iterated intersections. Additionally, any boundary components must also satisfy all transversality conditions. We always mean transversality in this generalized sense. 
}

\begin{remark}
    Leinster explains in \cite[Section 3.1]{leinster2000homotopy} that homotopy commutative algebras in a cartesian monoidal category is equivalent to a special $\Gamma$-object in the sense of Segal \cite{Se74}. That is, a functor $A:\mathrm{Fin}_*\to \mathrm{C}$ satisfying the {\bf Segal condition}: that $A[p]\to A[1]^p$ is an equivalence. This point of view is orthogonal to our discussion here since the our algebras take values in chain complexes whose operation $\otimes$ is not the cartesian product.
\end{remark}

\begin{remark}
    Leinster also defines the notion of a homotopy associative (noncommutative) algebras. Let $\mathrm{Ord}$ denote the category of finite ordinals and weakly order-preserving maps. This is (nonsymmetric) monoidal under the ordinal sum $\oplus: \mathrm{Ord}\times \mathrm{Ord}\to \mathrm{Ord}$, with unit $[0]=\emptyset$. A homotopy associative algebra in $(\mathrm{C},\otimes ,\mathbb{I})$ is a colax monoidal functor $A^\bullet:(\mathrm{Ord},\oplus,[0])\to (\mathrm{C},\otimes ,\mathbb{I})$ such that the `colax' structure maps $A(S\oplus T)\to A(S)\otimes A(T)$ and $A([0])\to \mathbb{I}$ are equivalences.
    
    A homotopy commutative algebra restricts to a homotopy associative algebra through the natural inclusion $(\mathrm{Ord},\oplus,[0])\into (\mathrm{Fin},\oplus,[0])$ that forgets the internal order on $[p] = \{1<2<\ldots<p\}$.
    Also, an associative algebra $A\otimes A\to A$ defines a homotopy associative algebra by $A^{[p]}=A^{\otimes p}$ and using the multiplication on $A$.

    We collectively refer to homotopy commutative and associative algebras as {\bf homotopy algebras}.
\end{remark}

\subsection{The bar construction}

One can define the bar construction of a homotopy algebra, as follows, generalizing the classical notion of bar construction discussed in \cref{S: bar cohomology}. First, we require the notion of an augmentation.

\begin{definition}[Augmentation]
     Let $\mathbb{I}^\bullet$ be the homotopy commutative algebra for which $\mathbb{I}^{[p]}$ is the unit $\mathbb{I}$ 
     for all $[p]$ and whose structure maps are given by the unit isomorphism $\mathbb{I}\cong \mathbb{I}^{\otimes p}$. An {\bf augmentation} on a homotopy algebra $A^\bullet$ is a natural transformation $\epsilon_A:A^\bullet \to \mathbb{I}^\bullet$.
    An {\bf augmented homotopy algebra} is a homotopy algebra $A^\bullet$ together with an augmentation $\epsilon_A$.

    A {\bf homomorphism} of augmented homotopy algebras  is a natural transformation $F^\bullet: A^\bullet\to B^\bullet$ of colax monoidal functors that respects the augmentations by $\epsilon_A = \epsilon_B\circ F$.
\end{definition}

The bar construction is defined using the following extension of the ordinal category. 

\begin{definition}
     Let $\mathrm{Ord}_{\pm}$ denote the category of finite ordinals with fixed extremal elements $\pm \infty$, where morphisms $S_{\pm}\to T_{\pm}$ are order preserving and must preserve the extremal elements. For $p\geq 0$ let $[p]_{\pm} = \{-\infty<1<\cdots<p<\infty\}$, with $[0]_{\pm}=\{-\infty<\infty\}$; and let $[p]\subset [p]_{\pm}$ denote the subset of finite elements. 
\end{definition}

With this category, the classical bar construction extends to the setting of homotopy algebras.

\begin{definition}[Simplicial bar construction of homotopy algebras]
    Let $A^\bullet: \mathrm{Ord}\to \mathrm{C}$ be a homotopy associative algebra with an augmentation $\epsilon: A^\bullet\to \mathbb{I}^\bullet$. The {\bf simplicial bar construction} of $A^\bullet$ is defined as 
    \[
    \BBar(A)_\bullet: \mathrm{Ord}_{\pm} \to \mathrm{C} \quad \text{ by }\quad \BBar(A)_\bullet: [p]_{\pm} \mapsto A^{[p]},
    \]
    where a morphism $f:[p]_{\pm}\to [q]_{\pm}$ acts by the composition of
    \[
    A^{[p]}\xrightarrow{\sim} A^{f^{-1}(-\infty)\cap [p]}\otimes A^{f^{-1}([q])} \otimes A^{f^{-1}(\infty)\cap [p]} \xrightarrow{\epsilon\otimes \id\otimes \epsilon} \mathbb{I} \otimes A^{f^{-1}([q])}\otimes \mathbb{I} \cong A^{f^{-1}([q])} \xrightarrow{f_*} A^{[q]}.
    \]
    Here we recall that $[0]=\emptyset$ so that $A^\emptyset=A^{[0]}$.
\end{definition}

In \cite[Section 3.1]{leinster2000homotopy}, Leinster then calls it a ``distracting coincidence'' that the category $\mathrm{Ord}_{\pm}$ is isomorphic to the opposite simplex category $\Delta^{\mathrm{op}}$, so $\BBar(A)_\bullet$ turns out to be a simplicial object! 
Explicitly, if we write $[p]_-$ for the ordered set $\{0<1<\cdots<p\}$ in $\Delta^{\mathrm{op}}$, this corresponds under the isomorphism to $\{-\infty<1<\cdots < p< \infty\}$ in $\mathrm{Ord}_{\pm}$; see \cite{leinster2000homotopy,JoyalTheta}. 
With this notation, we see that
\begin{equation}\label{E: Bar}
\BBar(A)_\bullet:[p]_-\mapsto A^{[p]}
\end{equation}
is a simplicial object in the category of chain complexes. Its degeneracy and internal face maps given by the homotopy algebra structure on $A$ and with extremal face maps given by the augmentation
\[
\partial_0:A^{[p]} \xrightarrow{\sim} A\otimes A^{[p-1]} \xrightarrow{\epsilon\otimes \id} A^{[p-1]} \text{ and } \partial_p:A^{[p]} \xrightarrow{\sim} A^{[p-1]}\otimes A \xrightarrow{\id\otimes \epsilon} A^{[p-1]}.
\]
As the formula \eqref{E: Bar} uses the natural number $p$ in two places, we will simply write $[p]\mapsto A^{[p]}$ below when in the simplicial context.  One should be mindful, however, that the notation $[p]$ has different meanings depending on whether it is treated as an object in $\mathrm{Ord}, \mathrm{Ord}_{\pm}$ or $\Delta^{\mathrm{op}}$.
\begin{remark}
    Given a left- and a right- $A^\bullet$-modules, which we will not define here, one can analogously construct the simplicial two-sided bar construction $\BBar(M,A,N)_\bullet:\Delta_{\pm}\to \mathrm{C}$. Now, the element $-\infty\in [p]_{\pm}$ accounts for the right-action of $A$ on $M$ and $\infty\in [p]_{\pm}$ accounts for the left-action on $N$. The definition we gave above is that of the special case $\BBar(\mathbb{I},A,\mathbb{I})_\bullet$, which is substantially easier to define succinctly.
\end{remark}

\begin{definition}
    When $\mathrm{C}$ admits geometric realizations, define the {\bf bar construction} of $A$ to be the object $\BBar(A) := |\BBar(A)_\bullet|$.
\end{definition}

For a homotopy algebra in the category of cochain complexes, an explicit model for $\BBar(A)$ is given by first replacing the simplicial object by its complex of normalized chains, which in this case is the double complex 
\[
\BBar(A)^{-p,\bullet} = A^{[p]}/\sum Im(\sigma_i),
\]
and then computing its total complex; see \cite{Ar26}. Here the ``horizontal'' differential is given by the alternating sum of the simplicial face maps, which correspond to products of neighboring tensor factors. For an associative algebra $A$, this agrees almost exactly with \cref{def:bar complex} of $\BBar(A)$;
the appearance there of $\bar A$ is precisely due to the quotient by the degeneracies. The only disagreement comes down to the signs in the differentials, but this can be remedied by a chain isomorphism consisting of a ``sign twist'' that multiplies each element by a sign depending on its indices. We assume such signs have been incorporated in the following discussion; see \cite{GBF47} for details. 

The double complex $\BBar(-)$ is a coassociative coalgebra, generalizing this structure on the bar complex for strictly associative algebras, with coproduct
\[
\Delta: \BBar(A)^{-p,\bullet} = A^{[p]}\to \bigoplus_{i+j=p} A^{[i]}\otimes A^{[j]} = \bigoplus_{i+j=p} \BBar(A)^{-i,\bullet}\otimes \BBar(A)^{-j,\bullet}
\]
defined by summing over the equivalences $A^{[p]}\to A^{[i]}\otimes A^{[j]}$ induced by the decomposition $[i]\oplus[j] \cong [p]$. This generalizes the coproduct \eqref{eq:tensor coproduct} from the strictly associative setting.  When the maps $A^{[p]}\to A^{[i]}\otimes A^{[j]}$ are inclusions, this is simply  deconcatenation.

\begin{definition}\label{def:geometric bar constr}
    For a smooth manifold $M$ without boundary, 
    let $\BBar_\Gamma(M;R)$ denote the bar construction of its geometric cochains with coefficients in $R$ given by $\BBar(C^*_\Gamma(M)\otimes R)$.
\end{definition}

One place where this bar construction makes direct contact with the associative one is in the weight spectral sequence. Indeed, since for every $p$ there is an equivalence $A^{[p]}\simeq A^{\otimes p}$, the two bar constructions are very close. In fact, all of the constructions and arguments made in \cref{S: bar cohomology} apply to homotopy algebras like geometric cochains. In particular, we have the following.

\begin{proposition}\label{prop:geometric weight spectral sequence}
    Let $A$ be a homotopy commutative algebra in cochain complexes over $R$. Then $H^*(A):=H^*(A^{[1]})$ has a unique structure as a strictly commutative algebra compatible with the homotopy commutative structure on $A$.
    
    Moreover, when both $A$ and $H^*(A)$ are flat as $R$-modules and bounded-below, the spectral sequence $\BBar(A)$ associated to its filtration by columns has
    \[
    E_1 \cong \BBar(H^*(A)) \implies H^*(\BBar(A)).
    \]
\end{proposition}
\begin{proof}
    Make $H^*(A)$ into an algebra with the product
    \[
    H^*(A)^{\otimes 2} \to H^*(A^{\otimes 2}) \xleftarrow\sim H^*(A^{[2]}) \to H^*(A),
    \]
    which is uniquely defined by inverting the middle isomorphism. The rightmost map is that induced by the unique map $[2] \to [1]$ in $\mathrm{Fin}$. Since all maps in this diagram are invariant under the swap of the two tensor factors, this product is strictly commutative.

    Now assume that $A$ and $H^*(A)$ are flat and filter $\BBar(A)$ by its columns: $F_p = \bigoplus_{p'\leq p} A^{[p']}$. The associated spectral sequence has $E_1^{-p,*} = H^*(A^{[p]})$. By the Leinster condition and since $A$ and $H^*(A)$ are assumed flat, we have by the K\"unneth Theorem natural isomorphisms $H^\bullet(A^{[p]}) \cong H^\bullet(A)^{\otimes p}$, and since the product on $H^*(A)$ is defined via the homotopy algebra structure maps, the $d_1$-differential is exactly the bar differential of $\BBar(H^*(A))$. The result now follows by standard spectral sequence machinery.
\end{proof}

We next recall the elegant feature of homotopy commutative algebras first observed by Segal \cite{Se74} in his work on $\Gamma$-spaces, namely that in appropriate categories the bar construction of a homotopy commutative algebra is again a homotopy commutative algebra and so the construction can be iterated indefinitely. In our context, the relevance of this fact is that $H^*(\BBar(C^*_\Gamma(-;R)))$  is a strictly commutative algebra, compatible with the coproduct.
\begin{proposition}\label{prop:geometric bar is homotopy commutative}
    For any homotopy commutative algebra $A^\bullet$ in the category $(\mathrm{Ch}, \otimes, R)$ of bounded-below cochain complexes of flat $R$-modules with equivalences given by quasi-isomorphisms, the complex $\BBar(A)$ is a homotopy commutative algebra in the category of coassociative coalgebras in $\mathrm{Ch}$. In particular, $H^*(\BBar(A))$ is a filtered commutative algebra with respect to the weight filtration.
\end{proposition}
\begin{proof}
    The key observation, essentially due to Segal \cite{Se74}, is that the cartesian product $\mathrm{Fin}\times \mathrm{Fin}\to \mathrm{Fin}$ makes every $\mathrm{Fin}$-object $A^\bullet$ into a $\mathrm{Fin}\times\mathrm{Fin}$-object $A^{\bullet\times \bullet}$, or a $\mathrm{Fin}$-object in the category of $\mathrm{Fin}$-objects, with underlying object $A^{1\times \bullet}\cong A^{\bullet}$. Since the cartesian product commutes with disjoint unions in either entry, the colax monoidal structure on $A^\bullet$ gives two commuting colax monoidal structures on $A^{\bullet\times\bullet}$, and both satisfy the Leinster condition of homotopy equivalence.

    We use the first coordinate of $A^{\bullet\times\bullet}$ to perform the bar construction. In other words, for any fixed $[p]\in \mathrm{Fin}$, we construct a simplicial object in $\mathrm{Ch}$ by
    \[
    (\BBar(A)^{[p]})_{\bullet} \colon [r]_{\pm} \mapsto A^{[r]\times [p]}.
    \]
Applying geometric realization gives an obect of $\mathrm{Ch}$ we will write $\BBar(A)^{[p]}$, and treating $[p]$ as a parameter in $\mathrm{Fin}$, $\BBar(A)^{\bullet}$ is a functor from $\mathrm{Fin}$ to $\mathrm{Ch}$. We claim this is itself a homotopy algebra. 
    
To check Leinster's conditions, we observe that $\BBar(A)^{[p+q]}$ is the realization of the simplicial object 
\[
[r]_{\pm} \mapsto A^{[r]\times [p+q]} \cong A^{[r]\times ([p]\oplus [q])}\cong A^{([r]\times [p])\oplus([r]\times [q])},
\]
which by assumption maps by a natural equivalence to  \[A^{[r]\times [p]} \otimes A^{[r]\times [q]}.\]
Note that $A^{\bullet\times [p]} \otimes A^{\bullet\times [q]}$ is the product simplicial object in $\mathrm{Ch}$. 
So, applying realization with respect to the first coordinate, this level-wise equivalence becomes an equivalence  \[\BBar(A)^{[p+q]} \xr{\sim} |(\BBar(A)^{[p]})_\bullet \otimes (\BBar(A)^{[q]})_\bullet| \xr{\sim}  |\BBar(A)^{[p]}| \otimes |\BBar(A)^{[q]}|,\]
using the Alexander-Whitney map and the Eilenberg-Zilber Theorem for the last equivalence.

For the case $[p]=[0]$, we have the constant simplicial object with \[(\BBar(A)^{[0]})_{[r]}\cong A^{[r]\times [0]} \cong A^{[0]},\]  
so the realization is equivalent to $A^{[0]}$, which by assumption is equivalent to the unit cochain complex $R$.

    The coalgebra structure was defined by summing over the equivalence $A^{[q]}\to A^{[i]}\otimes A^{[q-i]}$. To check that this structure commutes with the $\mathrm{Fin}$-structure maps it suffices to observe that for every map $[p]\to [p']$ in $\mathrm{Fin}$, the diagram
    \[
    \xymatrix{
    (\BBar(A)^{[p]})_{[r]} \ar@{=}[r] \ar[d] & A^{[r]\times [p]} \ar[r]\ar[d] & A^{[i]\times [p]} \otimes A^{[r-i]\times [p]} \ar[d] \\
    (\BBar(A)^{[p']})_{[r]} \ar@{=}[r] & A^{[r]\times [p']} \ar[r] & A^{[i]\times [p']} \otimes A^{[r-i]\times [p']}
    }
    \]
    commutes. This can be rephrased as saying that the bar construction is a coalgebra in whatever category it is being computed, and since $A^{\bullet\times\bullet}$ is a homotopy commutative algebra in the category of homotopy commutative algebras, the coproduct respects the residual product structure encoded by the $\mathrm{Fin}$-object.

    It follows by the previous lemma that $H^*(\BBar(A))$ has a commutative algebra structure, which respects the weight filtration by construction.
\end{proof}

Lastly, we wish to compare the bar constructions of different homotopy algebras.
\begin{proposition}
    A homomorphism $F^\bullet: A^\bullet\to B^\bullet$ be of augmented homotopy (commutative) algebras induces a homomorphisms of coalgebras (and homotopy commutative algebras)
    \[
    \BBar(F): \BBar(A)\to \BBar(B).
    \]
    Furthermore, when all $A^{[p]}$ and $B^{[p]}$ are bounded below complexes of flat $R$-modules, the map $\BBar(F)$ is a quasi-isomorphism whenever the map of underlying objects $F^{[1]}:A^{[1]}\to B^{[1]}$ is one.
\end{proposition}
\begin{proof}
    A natural transformation $F$ that preserves the respective augmentations obviously induces a natural transformation of simplicial object $\BBar(F)_\bullet:\BBar(A)_\bullet \to \BBar(B)_\bullet$  The induced map on geometric realizations is the proposed map $\BBar(F)$, which clearly respects the coalgebra and homotopy commutative algebra structures.

    Suppose now that $F^{[1]}:A\to B$ is a quasi-isomorphism and that all $A^{[p]}$ and $B^{[p]}$ are bounded-below and flat over $R$. Then for every $p\geq 0$ the map $F^{\otimes p}: A^{\otimes p}\to B^{\otimes p}$ is a quasi-isomorphism and thus the commutative diagram
    \[
    \xymatrix{
    A^{[p]} \ar[r]^{F^{[p]}} \ar[d] & B^{[p]} \ar[d] \\
    A^{\otimes p} \ar[r]^{F^{\otimes p}} & B^{\otimes p}
    }
    \]
    has all but the top arrow be a quasi-isomorphism. But this already implies that every component $F^{[p]}$ of the natural transformation $\BBar(F)_\bullet$ is a quasi-isomorphism. Since geometric realization takes level-wise equivalences to equivalences, the total map $\BBar(F)$ is a quasi-isomorphism.
\end{proof}
For the case of geometric cochains, these are flat over any ring, taking 
\[
C^*_\Gamma(-;R) := C^*_\Gamma(-)\otimes_\Z R.
\]
\begin{lemma}\label{L: goeometric flat}
    The complex of geometric cochains with coefficients in $R$ consists of flat $R$-modules.
\end{lemma}
\begin{proof}
    As shown in \cite{FMS-foundations}, the complex $C^*_\Gamma(-)$ is torsion-free over $\Z$, so it is flat over $\Z$ . The claim follows since extension of scalars $(-)\otimes_\Z R$ preserves flatness.
\end{proof}

\begin{corollary}[Proper topological invariance]\label{cor:topological invariance of geomteric bar}
    Suppose $f:M_0\to M_1$ is a smooth proper map of smooth manifolds without boundary, and let $i_f:C^*_{\Gamma,f}(M_1;R)\into C^*_{\Gamma}(M_1;R)$ be the quasi-isomorphic subcomplexes consisting of cochains transverse to $f$. Then the maps on geometric cochains induced the inclusion $i_f$ and by transverse pullback of cochains $f^*:C^*_{\Gamma,f}(M_1;R)\to C^*_\Gamma(M_0;R)$ induces a natural zig-zag of maps
    \[\BBar_{\Gamma}(M_1;R) \xleftarrow{\BBar_\Gamma(i_f)}\BBar(C^*_{\Gamma,f}(M_1;R))\xrightarrow{\BBar(f)} \BBar_\Gamma(M_0;R),\]
    where the left-pointing map is a quasi-isomorphism. In particular, $f$ induces a well-defined pullback on bar cohomology.
    
    Furthermore, if the cohomology groups of $M_0$ and $M_1$ are finitely generated and $f$ induces an isomorphism on geometric cohomology with coefficients in $R$, then $\BBar(f)$ is a quasi-isomorphism as well.
\end{corollary}

\begin{proof}
    We refer the reader to \cite[Section 7.4]{FMS-foundations} for details concerning $C^*_{\Gamma,f}(M_1;R)$ and the pullback map $f^*:C^*_\Gamma(M_1;R)\to C^*_\Gamma(M_0;R)$ and to \cite{GBF47} for details concerning $\BBar(C^*_{\Gamma,f}(M_1;R))$ and the fact that $f^*$ extends to a homomorphism of homotopy commutative algebras.
    Together, these provide a pullback of homotopy commutative algebras and so induce a map of bar complexes, as claimed. By the last proposition, since $i_f$ is a quasi-isomorphism, its induced map on bar constructions is also one.

    Lastly, if $M_0$ and $M_1$ have finitely generated homology, then \cite[Corollary 6.21 and Theorem 6.29]{FMS-foundations} show that there exists a natural isomorphism $H^*(C^*_\Gamma(-;R))\cong H^*_{Sing}(-;R)$, and thus an $R$-homology isomorphism induces a quasi-isomorphism on geometric cochains. In such a case, again, the last proposition shows that $\BBar(f^*)$ is a quasi-isomorphism as well.
\end{proof}
In light of the quasi-isomorphism invariance, to compare the singular and geometric bar constructions it suffices to find quasi-isomorphisms of partial algebras $C^*_\Gamma(-)\xrightarrow{\sim} A(-) \xleftarrow{\sim}  C^*_{Sing}(-)$. We learned through private communication that, restricting to $\Z/2\Z$-coefficients, such a comparison is constructed in Pizzi's forthcoming dissertation \cite{Pizzi}. Hence the invariants obtained through the two models are naturally equivalent over $\Z/2\Z$.

    \bibliographystyle{alpha}
    \bibliography{bibliography}
\end{document}